\documentclass[11pt]{amsart}
\usepackage{amsmath,amssymb,amsthm,amscd,fancyhdr}

\numberwithin{equation}{section}

\RequirePackage[colorlinks,linkcolor=blue,filecolor = blue,citecolor = blue, urlcolor =blue]{hyperref}

\newtheorem{theorem}{Theorem}[section]
\newtheorem{prop}[theorem]{Proposition}
\newtheorem{lemma}[theorem]{Lemma}
\theoremstyle{definition}
\newtheorem{remark}[theorem]{Remark}

\newcommand{\eps}{\varepsilon}
\renewcommand{\epsilon}{\varepsilon}

\renewcommand{\le}{\leqslant}
\renewcommand{\ge}{\geqslant}

\newcommand{\ev}{{\sf ev}}

\newcommand{\eqdef}{\stackrel{\text{\tiny def}}{=}}
\newcommand{\Q}{\mathbb{Q}}       % set of rational numbers
\newcommand{\F}{\mathbb{F}}      % set of integers
\renewcommand{\P}{\mathbb{P}}
\renewcommand{\L}{{\mathcal L}}
\newcommand{\C}{{\mathcal C}}
\renewcommand{\deg}{{\rm deg}}
\renewcommand{\dim}{{\rm dim}}
\renewcommand{\O}{{\mathcal O}}
\newcommand{\ac}{{\rm ac}}
\renewcommand{\b}{\beta}

\newcommand{\mv}[1]{{\mathbf {#1}}}

\title{{\bf Artin automorphisms, Cyclotomic function fields, and \\ Folded list-decodable codes}}

\author{Venkatesan Guruswami}

\thanks{\noindent Research supported in part by NSF CCF-0343672, a David and Lucile Packard Fellowship, and NSF grant CCR-0324906 to the IAS}
\address{Department of Computer Science and
  Engineering, University of Washington. Currently visiting the
  Computer Science Dept., Carnegie Mellon University. Some of
  this work was done when the author was a member in the School of Mathematics, Institute for Advanced Study.}
\email{venkat@cs.washington.edu}

\date{}
\begin{document}
\maketitle
\thispagestyle{empty}

\begin{abstract}
  Algebraic codes that achieve list decoding capacity were recently
constructed by a careful ``folding'' of the Reed-Solomon code. The
``low-degree'' nature of this folding operation was crucial to the list
decoding algorithm. We show how such folding schemes conducive to list
decoding arise out of the Artin-Frobenius automorphism at primes in
Galois extensions. Using this approach, we construct new folded
algebraic-geometric codes for list decoding based on cyclotomic
function fields with a cyclic Galois group. Such function fields are
obtained by adjoining torsion points of the Carlitz action of an
irreducible $M \in \F_q[T]$. The Reed-Solomon case corresponds to the
simplest such extension (corresponding to the case $M=T$). In the
general case, we need to descend to the fixed field of a suitable
Galois subgroup in order to ensure the existence of many degree one
places that can be used for encoding.

Our methods shed new light on algebraic codes and their list decoding,
and lead to new codes achieving list decoding capacity.
Quantitatively, these codes provide list decoding (and list
recovery/soft decoding) guarantees similar to folded Reed-Solomon
codes but with an alphabet size that is only polylogarithmic in the
block length. In comparison, for folded RS codes, the alphabet size is
a large polynomial in the block length. This has applications to
fully explicit (with no brute-force search) binary concatenated codes
for list decoding up to the Zyablov radius.

\end{abstract}

{\small
\tableofcontents
}
\newpage

\parindent 0mm
\parskip 3mm
\section{Introduction}

\subsection{Context and Motivation}
Recent progress in algebraic coding
theory~\cite{PV-focs05,GR-capacity} has led to the construction of
explicit codes over large alphabets that achieve list decoding
capacity --- namely, they admit efficient algorithms to correct close
to the optimal fraction $1-R$ of errors with rate $R$.  The algebraic
codes constructed in \cite{GR-capacity} are {\em folded} Reed-Solomon
codes, where the Reed-Solomon (RS) encoding
$(f(1),f(\gamma),\cdots,f(\gamma^{n-1}))$ of a low-degree polynomial
$f \in \F_q[T]$ is viewed as a codeword of length $N = n/m$ over the
alphabet $\F_q^m$ by identifying successive blocks of $m$
symbols. Here $\gamma$ is a primitive element of the field $\F_q$.

Simplifying matters somewhat, the principal algebraic engine behind
the list decoding algorithm in \cite{GR-capacity} was the identity
$f(\gamma T) \equiv f(T)^q \pmod {(T^{q-1} - \gamma)}$, and the fact
that $(T^{q-1} - \gamma)$ is irreducible over $\F_q$. This gave a
low-degree algebraic relation between $f(T)$ and $f(\gamma T)$ in the
residue field $\F_q[T]/(T^{q-1} - \gamma)$. This together with an
algebraic relation found by the ``interpolation step'' of the decoding
enabled finding the list of all relevant message polynomials $f(T)$
efficiently.

One of the main motivations of this work is to gain a deeper
understanding of the general algebraic principles underlying the
above folding, with the hope of extending it to more general
algebraic-geometric (AG) codes. The latter question is an interesting
algebraic question in its own right, but is also important for
potentially improving the alphabet size of the codes, as well as the
decoding complexity and output list size of the decoding
algorithm. (The large complexity and list size of the folded RS
decoding algorithm in \cite{GR-capacity} are a direct consequence of
the large degree $q$ in the identity relating $f(\gamma T)$ and
$f(T)$.)

An extension of the Parvaresh-Vardy codes~\cite{PV-focs05} (which were
the precursor to the folded RS codes) to arbitrary algebraic-geometric
codes was achieved in \cite{GP-mathcomp}. But in these codes the
encoding includes the evaluations of an additional function explicitly
picked to satisfy a low-degree relation over some residue field. This
leads to a substantial loss in rate. The crucial insight in the construction of 
folded RS codes was the fact that this additional function could
just be the closely related function $f(\gamma T)$ --- the image of
$f(T)$ under the automorphism $T \mapsto \gamma T$ of $\F_q(T)$.

\vspace{-1ex}
\subsection{Summary of our contributions}
We explain how folding schemes conducive to list decoding (such as the
above relation between $f(\gamma T)$ and $f(T)$) arise out of the {\em
  Artin-Frobenius automorphism} at primes in Galois extensions.  With
the benefit of hindsight, the role of such automorphisms in folding
algebraic codes is quite natural. In terms of technical contributions,
we use this approach to construct new list-decodable folded
algebraic-geometric codes based on {\em cyclotomic function fields}
with a cyclic Galois group. Cyclotomic function
fields~\cite{carlitz,hayes} are obtained by adjoining torsion points
of the Carlitz action of an irreducible $M \in \F_q[T]$. The
Reed-Solomon case corresponds to the simplest such extension
(corresponding to the case $M=T$). In the general case, we need to
descend to the fixed field of a suitable Galois subgroup in order to
ensure the existence of many degree one places that can be used for
encoding.  We establish some key algebraic lemmas that characterize
the desired subfield in terms of the appropriate generator $\mu$ in
the algebraic closure of $\F_q(T)$ and its minimal polynomial over
$\F_q(T)$. We then tackle the computational algebra challenge of
computing a representation of the subfield and its rational places,
and the message space, that is conducive for efficient encoding and
decoding of the associated algebraic-geometric code.

Our constructions lead to some substantial quantitative improvements
in the alphabet size which we discuss below in
Section~\ref{sec:long-intro}. We also make some simplifications in the
list decoding algorithm and avoid the need of a zero-increasing basis
at each code place (Lemma~\ref{lem:zero-count-1}). This, together with
several other ideas, lets us implement the list decoding algorithm in
polynomial time assuming {\em only} the natural representation of the
code needed for efficient encoding, namely a basis for the message
space. Computing such a basis remains an interesting question in
computational function field theory. Our description and analysis of
the list decoding algorithm in this work is self-contained, though it
builds strongly on the framework of the algorithms in
\cite{sudan-RS,PV-focs05,GP-mathcomp,GR-capacity}.

\vspace{-1ex}
\subsection{Galois extensions and Artin automorphisms}
We now briefly discuss how and why Artin-Frobenius automorphisms arise
in the seemingly distant world of list decoding.  In order to
generalize the Reed-Solomon case, we are after function fields whose
automorphisms we have a reasonable understanding of. Galois extensions
are a natural subclass of function fields to consider, with the hope
that some automorphism in the Galois group will give a low-degree
relation over some residue field. Unfortunately, the explicit
constructions of good AG codes are typically based on a tower of
function fields~\cite{gar-stich-1,gar-stich-2}, where each step is
Galois, but the whole extension is
not. (Stichtenoth~\cite{stich-galois} recently showed the existence of
a Galois extension with the optimal trade-off between genus and number
of rational places, but this extension is not, and cannot be, cyclic,
as we require.)

In Galois extensions $K/F$, for each place $A'$ in the extension field
$K$, there is a special and important automorphism called the
Artin-Frobenius automorphism (see, eg. \cite[Chap. 4]{marcus}) that
simply powers the residue of any (regular) function at that place.
The exponent or degree of this map is the norm of the place $A$ of $F$
lying below $A'$. Since the degree dictates the complexity of
decoding, we would like this norm to be small. On the other hand, the
residue field at $A'$ needs to be large enough so that the message
functions can be uniquely identified by their residue modulo $A'$. The
most appealing way to realize this is if the place $A$ is inert, i.e.,
has a unique $A'$ lying above it. However, this condition can only
hold if the Galois group is cyclic, a rather strong restriction. For
example, it is known~\cite{FPS} that even abelian extensions must be
{\em asymptotically bad}.

In order to construct AG codes, we also need to have a good control of
how certain primes split in the extension. For cyclotomic function
fields, and of course their better known number-theoretic counterparts
$\Q(\omega)$ obtained by adjoining a root of unity $\omega$, this
theory is well developed. As mentioned earlier, the cyclotomic
function field we use itself has very few  rational places. So we need
to descend to an appropriate subfield where many degree one places of
$\F_q(T)$ split completely, and develop some underlying theory
concerning the structure of this subfield.

The Artin-Frobenius automorphism\footnote{Following
  Rosen~\cite{rosen}, we will henceforth refer to the Artin-Frobenius
  automorphisms as simply Artin automorphisms. Many texts
  (eg. \cite{marcus}) actually refer to these as Frobenius
  automorphisms. Since the latter term is most commonly associated
  with automorphism $x \mapsto x^q$ of $\F_{q^m}$, we prefer the term
  Artin automorphism to refer to the general notion that applies to
  all Galois extensions. The association of a place with its
  Artin-Frobenius automorphism is called the Artin map.}
is a fundamental notion in
algebraic number theory, playing a role in Chebatorev density theorem
and Dirichlet's theorem on infinitude of primes in arithmetic
progressions, as well as quadratic and more general reciprocity laws.
We find it rather intriguing that this notion ends up playing an
important role in algorithmic coding theory as well.

\subsection{Long codes achieving list decoding capacity and explicit
  binary concatenated codes}
\label{sec:long-intro}
Quantitatively, our cyclotomic function field codes achieve list
decoding (and list recovery) guarantees similar to folded RS codes but
with an alphabet size that is only {\em polylogarithmic} in the block
length. In comparison, for folded RS codes, the alphabet size is a
large polynomial in the block length.  We note that Guruswami and
Rudra~\cite{GR-capacity} also present capacity-achieving codes of rate
$R$ for list decoding a fraction $(1-R-\eps)$ of errors with alphabet
size $|\Sigma| = 2^{(1/\eps)^{O(1)}}$, a fixed constant depending only
on $\eps$. But these codes do not have the strong ``list recovery''
(or more generally, soft decoding) property of folded RS codes.

Our codes inherit the powerful list recovery property of folded RS
codes, which makes them very useful as outer codes in concatenation
schemes. In fact, due to their small alphabet size, they are even
better in this role. Indeed, they can serve as outer codes for a
family of concatenated codes list-decodable up to the Zyablov radius,
{\em with no brute-force search} for the inner codes. This is the
first such construction for list decoding. It is similar to the
``Justesen-style'' explicit constructions for rate vs. distance from
\cite{justesen,shen}, except even easier, as one can use the ensemble
of {\em all linear codes} instead of the succinct Wozencraft ensemble
at the inner level of the concatenated scheme.

\subsection{Related work}
\label{subsec:related}
Codes based on cyclotomic function fields have been considered
previously in the literature. Some specific (non-asymptotic)
constructions of function fields with many rational places over small
fields $\F_q$ ($q \le 5$) appear in \cite {NX96,NX97}.  Cyclotomic
codes based on the action of polynomials $T^a$ for small $a$ appear in
\cite{quebbemann}, but decoding algorithms are not discussed for these
codes, nor are these extensions cyclic as we require.  Our approach is
more general and works based on the action of an arbitrary irreducible
polynomial. Exploiting the Artin automorphism of cyclotomic fields for an algorithmic
purpose is also new to this work.

Independent of our work, Huang and Narayanan~\cite{HN-folded-ag} also
consider AG codes constructed from Galois extensions, and observe how
automorphisms of large order can be used for folding such codes. To
our knowledge, the only instantiation of this approach that improves
on folded RS codes is the one based on cyclotomic function fields from
our work. As an alternate approach, they also propose a decoding
method that works with folding via automorphisms of small order. This
involves computing several coefficients of the power series expansion
of the message function at a low-degree place. Unfortunately, piecing
together these coefficients into a function could lead to an
exponential list size bound. The authors suggest a heuristic
assumption under which they can show that for a {\em random} received word,
the expected list size and running time are polynomially bounded.

\section{Background on Cyclotomic function fields}
Some basic preliminaries on function fields, valuations and places,
Galois extensions, decomposition of primes, Artin-Frobenius
automorphism, etc. are discussed in Appendix~\ref{app:alg-prelims}. In
this section, we will focus on background material concerning
cyclotomic function fields. These are the function-field analog of the
classic cyclotomic number fields from algebraic number theory. This
theory was developed by Hayes~\cite{hayes} in 1974 building upon ideas
due to Carlitz~\cite{carlitz} from the late 1930's. The objective was
to develop an explicit class field theory classifying all abelian
extensions of the rational function field $\F_q(T)$, analogous to
classic results for ${\mathbb Q}$ and imaginary quadratic extensions
of ${\mathbb Q}$. The common idea in these results is to allow a ring
of ``integers'' in the ground field to act on part of its algebraic
closure, and obtain abelian extensions by adjoining torsion points of
this action.  We will now describe these extensions of $\F_q(T)$.

Let $T$ be an indeterminate over the finite field $\F_q$. Let $R_T =
\F_q[T]$ denote the polynomial ring, and $F = \F_q(T)$ the field of
rational functions. Let $F^\ac$ be a fixed algebraic closure of $F$.
Let ${\rm End}_{\F_q}(F^\ac)$ be the ring of $\F_q$-endomorphisms of
$F^\ac$, thought of as a $\F_q$-vector space.  We consider two special
elements of ${\rm End}_{\F_q}(F^\ac)$: (i) the Frobenius automorphism
$\tau$ defined by $\tau(z) = z^q$ for all $z \in F^\ac$, and (ii) the
map $\mu_T$ defined by $\mu_T(z) = Tz$ for all $z \in F^\ac$.  The
substitution $T \rightarrow \tau+\mu_T$ yields a ring homomorphism
from $R_T$ to ${\rm End}_{\F_q}(F^\ac)$ given by: $f(T) \mapsto
f(\tau+\mu_T)$. Using this, we can define the {\em Carlitz action} of
$R_T$ on $F^\ac$ as follows: For $M \in R_T$,
\[ C_M(z) = M(\tau+\mu_T)(z) \qquad \mbox{for all } z \in F^\ac \ . \]
This action endows $F^\ac$ the structure of an $R_T$-module, which is
called the Carlitz module. For a nonzero polynomial $M \in R_T$,
define the set
\[ \Lambda_M = \{ z \in F^\ac \mid C_M(z) = 0 \} \ , \]
%$\Lambda_M = \{ z \in F^\ac \mid C_M(z) = 0 \}$,
to consist of
the $M$-torsion points of $F^\ac$, i.e., the elements annihilated by
the Carlitz action of $M$ (this is also the set of zeroes of the
polynomial $C_M(Z) \in R_T[Z]$). Since $R_T$ is commutative,
$\Lambda_M$ is in fact an $R_T$-submodule of $F^\ac$. It is in fact a
cyclic $R_T$-module, naturally isomorphic to $R_T/(M)$.

The cyclotomic function field $F(\Lambda_M)$ is obtained by adjoining
the set $\Lambda_M$ of $M$-torsion points to $F$.
\footnote{It is
  instructive to compare this with the more familiar setting of
  cyclotomic number fields. There, one lets ${\mathbb Z}$ act on the
  multiplicative group $({\mathbb Q}^{\ac})^*$ with the endomorphism
  corresponding to $n \in {\mathbb Z}$ sending $\zeta \mapsto \zeta^n$
  for $\zeta \in {\mathbb Q}^{\ac}$. The $n$-torsion points now equal
  $\{ \zeta \in {\mathbb Q}^{\ac} \mid \zeta^n = 1\}$, i.e., the
  $n$'th roots of unity. Adjoining these gives the various cyclotomic
  number fields.} 
The following result from \cite{hayes} summarizes some fundamental
facts about cyclotomic function fields, stated for the special case
when $M$ is irreducible (we will only use such extensions). Proofs can
also be found in the graduate texts \cite[Chap. 12]{rosen} or
\cite[Chap. 12]{salvador}. In what follows, we will often use the
convention that an irreducible polynomial $P \in R_T$ is identified
with the place of $F$ which is the zero of $P$, and also denote this
place by $P$. Recall that these are all the places of $F$, with the
exception of the place $P_\infty$, which is the unique pole of $T$.

\newcommand{\Gal}{{\rm Gal}}

\begin{prop}
\label{prop:basic-cyclo}
  Let $M \in R_T$ be a nonzero degree $d$ monic polynomial that is irreducible
  over $\F_q$. Let $K = F(\Lambda_M)$. Then
\begin{enumerate}
\item $C_M(Z)$ is a separable polynomial in $Z$ of degree $q^d$ over
  $R_T$, of the form $\sum_{i=0}^d [M,i] Z^{q^i}$ where the degree of
  $[M,i]$ as a polynomial in $T$ is $q^i(d-i)$. The polynomial
  $\psi_M(Z) = C_M(Z)/Z$ is irreducible in $R_T[Z]$. The field $K$ is
  equal to the splitting field of $\psi_M(Z)$, and is generated by any
  nonzero element $\lambda \in \Lambda_M$, i.e., $K = F(\lambda)$.
\item $K/F$ is a Galois extension of degree $(q^d-1)$ and $\Gal(K/F)$
  is isomorphic to $(R_T/(M))^*$, the cyclic multiplicative group of
  units of the field $R_T/(M)$. The Galois automorphism $\sigma_N$
  associated with $\bar{N} \in (R_T/(M))^*$ is given by
  $\sigma_N(\lambda) = C_N(\lambda)$. \\
  The Galois automorphisms commute with the Carlitz action: for any
  $\sigma \in \Gal(K/F)$ and $A \in R_T$, $\sigma(C_A(x)) =
  C_A(\sigma(x))$ for all $x \in K$.
\item If $P \in R_T$ is a monic irreducible polynomial different from
  $M$, then the Artin automorphism at the place $P$ is equal to
  $\sigma_P$.
\item The integral closure of $R_T$ in $F(\lambda)$ equals $R_T[\lambda]$.
\item The genus $g_M$ of $F(\Lambda_M)$ satisfies $2g_M-2 = d(q^d-2) - \frac{q}{q-1} (q^d-1)$.
\end{enumerate}
\end{prop}

The splitting behavior of primes in the extension $F(\Lambda_M)/F$
will be crucial for our construction. We record this as a separate
proposition below.
\begin{prop}
\label{prop:PS}
Let $M \in R_T$, $M\neq 0$, be a monic, irreducible polynomial of degree $d$. 
\begin{enumerate}

\item (Ramification at $M$) The place $M$ is totally ramified in the extension
  $F(\Lambda_M)/F$. If $\lambda \in \Lambda_M$ is a root of $C_M(z)/z$
  and $\tilde{M}$ is the unique place of $F(\Lambda_M)$ lying above
  $M$, then $\lambda$ is a $\tilde{M}$-prime element, i.e.,
  $v_{\tilde{M}}(\lambda) = 1$.

\item (Ramification at $P_\infty$) The infinite place $P_\infty$ of
  $F$, i.e., the pole of $T$, splits into $(q^d-1)/(q-1)$ places of
  degree one in $F(\Lambda_M)/F$, each with ramification index
  $(q-1)$. Its decomposition group equals $\F_q^*$.

\item (Splitting at other places) If $P \in R_T$ is a monic
  irreducible polynomial different from $M$, then $P$ is unramified in
  $F(\Lambda_M)/F$, and splits into $(q^d-1)/f$ primes of degree $f \cdot
  \deg(P)$ where $f$ is the order of $P$ modulo $M$ (i.e., the
  smallest positive integer $e$ such that $P^e \equiv 1 \pmod {M}$).

\end{enumerate}

\end{prop}

\section{Reed-Solomon codes as cyclotomic function field codes}
We now discuss how Reed-Solomon codes arise out of the simplest
cyclotomic extension $F(\Lambda_T)/F$. This serves both as a warm-up
for our later results, and as a method to illustrate that one can view
the folding employed by Guruswami and Rudra~\cite{GR-capacity} as
arising naturally from the Artin automorphism at a certain prime in
the extension $F(\Lambda_T)/F$.

We have $\Lambda_T = \{ u \in F^\ac \mid u^q + T u = 0 \}$. Pick a
nonzero $\lambda \in \Lambda_T$. By Proposition~\ref{prop:PS}, the
only ramified places in $F(\Lambda_T)/F$ are $T$, and the pole
$P_\infty$ of $T$. Both of these are totally ramified and have a
unique place above them in $F(\Lambda_T)$. Denote by
$Q_\infty$ the place above $P_\infty$ in $F(\Lambda_T)$.

We have $\lambda^{q-1} = -T$, so $\lambda$ has a pole of order one at
$Q_\infty$, and no poles elsewhere. The place $T+1$ splits completely
into $n=q-1$ places of degree one in $F(\Lambda_T)$. The evaluation of
$\lambda$ at these places correspond to the roots of $x^{q-1} = 1$,
i.e., to nonzero elements of $\F_q$. Thus the places above $T+1$ can
be described as $P_1, P_\gamma, \cdots, P_{\gamma^{q-2}}$ where
$\gamma$ is a primitive element of $\F_q$ and $\lambda(P_{\gamma^i}) =
\gamma^i$ for $i=0,1,\dots,q-2$.

For $k < q-1$, define ${\mathcal M}_k
= \{ \sum_{i=0}^{k-1} \beta_i \lambda^i \mid \beta_i \in
\F_q\}$. ${\mathcal M}_k$ has $q^k$ elements, each with at most $(k-1)$
poles at $Q_\infty$ and no poles elsewhere. Consider the $\F_q$-linear
map $E_{\rm RS}: {\mathcal M}_k \rightarrow \F_q^n$ defined as
\[ E_{\rm RS}(f) = \Bigl( f(P_1), f(P_\gamma), \cdots , f(P_{\gamma^{q-2}}) \Bigr)
\ . \]
Clearly the above just defines an $[n,k]_q$ Reed-Solomon code, consisting of evaluations of polynomials of degree $< k$ at elements of $\F^*_q$.

Consider the place $T+\gamma$ of $F$. The condition $(T+\gamma)^f
\equiv 1 \pmod {T}$ is satisfied iff $\gamma^f = 1$, which happens iff
$(q-1) | f$. Therefore, the place $T+\gamma$ remains inert in
$F(\Lambda_T)/F$. Let $A$ denote the unique place above $T+\gamma$
in $F(\Lambda_T)$. The degree of $A$ equals $q-1$.

The Artin automorphism at $A$, $\sigma_A$, is given by
$\sigma_A(\lambda) = C_{T+\gamma}(\lambda) = C_\gamma(\lambda) =
\gamma \lambda$.  Note that this implies $f(P_{\gamma^{i+1}}) =
\sigma_A(f)(P_{\gamma^i})$ for $0 \le i < q-2$.  By the property of
the Artin automorphism, we have $\sigma_A(f) \equiv f^q \pmod {A}$ for
all $f \in R_T[\lambda]$.  Note that this is same as the condition
$f(\gamma \lambda) \equiv f(\lambda)^q \pmod {(\lambda^{q-1} -
  \gamma)}$ treating $f$ as a polynomial in $\lambda$. This
corresponds to the algebraic relation between $f(X)$ and $f(\gamma X)$
in the ring $\F_q[X]$ that was used by Guruswami and
Rudra~\cite{GR-capacity} in their decoding algorithm, specifically in
the task of finding all $f(X)$ of degree less than $k$ satisfying
$Q(X,f(X),f(\gamma X)) = 0$ for a given $Q \in \F_q[X,Y,Z]$. In the
cyclotomic language, this corresponds to finding all $f \in
R_T[\lambda]$ with $< k$ poles at $Q_\infty$ satisfying
$Q(f,\sigma_A(f)) = 0$ for $Q \in R_T[\lambda](Y,Z)$. Since $\deg(A) =
q-1 \ge k$, $f$ is determined by its residue at $A$, and we know
$\sigma_A(f) \equiv f^q \pmod {A}$. Therefore, we can find all such
$f$ by finding the roots of the univariate polynomial $Q(Y,Y^q) \mod
A$ over the residue field ${\mathcal O}_A/A$.

\section{Subfield construction from cyclic cyclotomic function fields}
\label{sec:main-const}
In this section, we will construct the function field construction
that will be used for our algebraic-geometric codes, and establish the
key algebraic facts concerning it.  The approach will be to take
cyclotomic field $K=F(\Lambda_M)$ where $M$ is an irreducible of
degree $d > 1$ and get a code over $\F_q$.  But the only places of
degree $1$ in $F(\Lambda_M)$ are the ones above the pole $P_\infty$ of
$T$. There are only $(q^d-1)/(q-1)$ such places above $P_\infty$, which is much smaller
than the genus. So we descend to a subfield where many degree $1$
places split completely. This is done by taking a subgroup $H$ of
$(\F_q[T]/(M))^*$ with many degree $1$ polynomials and considering the
fixed field $E=K^H$. For every irreducible $N \in R_T$ such that
$\bar{N} = N \mod M \in H$, the place $N$ splits completely in the
extension $E/F$ (this follows from the fact that $C_N$ is the Artin
automorphism at the place $N$).  This technique has also been used in
the previous works \cite{quebbemann,NX96,NX97} mentioned in
Section~\ref{subsec:related}, though our approach is more general and
works with any irreducible $M$. The study of algorithms for cyclotomic
codes and the role played by the Artin automorphism in their list
decoding is also novel to our work.

\subsection{Table of parameters} Since there is an unavoidable surfeit
of notation and parameters used in this section and
Section~\ref{sec:code-const}, we summarize them for easy reference in
Appendix~\ref{app:params}.  

\subsection{Function field construction}

Let $\F_r$ be a subfield of $\F_q$. Let $M \in \F_r[T]$ be a
monic polynomial that is irreducible over $\F_q$ (note that we require
$M(T)$ to have coefficients in the smaller field $\F_r$, but demand
irreducibility in the ring $\F_q[T]$).  The following lemma follows
from the general characterization of when binomials $T^m - \alpha$ are
irreducible in $\F_q[T]$~\cite[Chap. 3]{LN-book}.
\begin{lemma}
\label{lem:binomial-irred}
Let $d \ge 1$ be an odd integer such that every prime factor of $d$ divides $(r-1)$ and  ${\rm gcd}(d, (q-1)/(r-1)) = 1$. Let $\gamma$ be a primitive element of $\F_r$. Then $T^d -\gamma \in \F_r[T]$ is irreducible in $\F_q[T]$.
\end{lemma}

A simple choice for which the above conditions are met is $r=2^a$,
$q=r^2$, and $d = r-1$ (we will need a more complicated choice for our
list decoding result in Theorem~\ref{thm:main-final}). For the sake of
generality as well as clarity of exposition, we will develop the
theory without making specific choices for the parameters, a somewhat
intricate task we will undertake in Section~\ref{sec:params}.

For the rest of this section, fix $M(T) = T^d - \gamma$ as guaranteed
by the above lemma.  We continue with the notation $F = \F_q(T)$, $R_T
= \F_q[T]$, and $K = F(\Lambda_M)$. Fix a generator $\lambda \in \Lambda_M$ of $K/F$ so that $K = F(\lambda)$. 

Let $G$ be the Galois group of $K/F$, which is isomorphic to the cyclic
multiplicative group $(\F_q[T]/(M))^*$. Let $H \subset G$ be the
subgroup $\F_q^* \cdot (\F_r[T]/(M))^*$. The cardinality of $H$ is
$(r^d-1) \cdot \frac{q-1}{r-1}$. Note that since $G$ is cyclic there
is a unique subgroup $H$ of this size. Indeed, if $\Gamma \in G$ is an
arbitrary generator of $G$, then $H = \{1,\Gamma^b,\Gamma^{2b}, \dots,
\Gamma^{q^d-1-b}\}$ where
\begin{equation}
\label{eq:def-of-b}
b = \frac{|G|}{|H|} = \frac{q^d-1}{r^d-1} \cdot \frac{r-1}{q-1} \ . 
\end{equation}

Let $A \in R_T$ be an arbitrary polynomial such that $A \mod M$ is a
generator of $(\F_q[T]/(M))^*$. We can then take $\Gamma$ so that
$\Gamma(\lambda) = C_A(\lambda)$. (We fix a choice of $A$ in the
sequel and assume that $A$ is pre-computed and known. We will later,
in Section~\ref{sec:hqp}, pick such an $A$ of appropriately large
degree.) Note that by part (2) of Proposition~\ref{prop:basic-cyclo},
the Galois action commutes with the Carlitz action and therefore
$\Gamma^j(\lambda) = C_{A^j}(\lambda)$ for all $j \ge 1$.  Thus
knowing the polynomial $A$ lets us compute the action of the
automorphisms of $H$ on any desired element of $K = F(\lambda)$.

 Let $E \subset K$ be the subfield of $K$ fixed by the
subgroup $H$, i.e., $E = \{x \in K \mid \sigma(x) = x ~\forall \sigma
\in H\}$. The field $E$ will be the one used to construct our
codes. We first record some basic properties of the extension $E/F$,
and how certain places decompose in this extension. 

\begin{prop}
\label{prop:basic-subfield}
For $E = F(\Lambda_M)^H$, the following properties hold:
\begin{enumerate}
\item $E/F$ is a Galois extension of degree $[E:F] = b$.
\item The place $M$ is the only ramified place in $E/F$, and it is totally ramified with a unique place (call it $M'$) above it in $E$.
\item The infinite place $P_\infty$ of $F$, i.e., the pole of $T$, splits completely into $b$ degree one places in $E$.
\item The genus $g_E$ of $E$ equals $\frac{d (b-1)}{2} + 1$.
\item For each $\beta \in \F_r$, the place $T -\beta$ of $F$ splits
  completely into $b$ degree one places in $E$.
\item If $A \in R_T$ is irreducible of degree $\ell \ge 1$ and $A \mod
  M$ is a primitive element of $R_T/(M)$, then the place $A$ is inert
  in $E/F$. The Artin automorphism $\sigma_A$ at $A$ satisfies
\begin{equation}
\label{eq:prop-of-artin-auto}
\sigma_A(x)  \equiv x^{q^\ell} \pmod {A'} 
\end{equation}
for all $x \in \O_{A'}$, where $A'$ is the unique place of $E$ lying above $A$.
\end{enumerate}
\end{prop}
\begin{proof}
  By Galois theory, $[E:F] = |G|/|H| = b$. Since $G$ is abelian, $E/F$
  is Galois with Galois group isomorphic to $G/H$. Since $E \subset
  K$, and $M$ is totally ramified in $K$, it must also be totally
  ramified in $E$. The only other place ramified in $K$ is $P_\infty$,
  and since $H$ contains the decomposition group $\F_q^*$ of
  $P_\infty$, $P_\infty$ must split completely in $E/F$.  

  The genus of $E$ is easily computed since $E/F$ is a tamely ramified
  extension~\cite[Sec. III.5]{stich-book}. Since only the place $M$ of
  degree $d$ is ramified, we have $2g_E-2 = d (b -1)$.

  Since $H \supset \F_r[T]$, for $\beta \in \F_r$, the Artin
  automorphism $\sigma_{T-\beta}$ of the place $T-\beta$ in $K/F$
  belongs to $H$. The Artin automorphism of $T-\beta$ in the extension
  $E/F$ is the restriction of $\sigma_{T-\beta}$ to $E$, which is
  trivial since $H$ fixes $E$. It follows that $T-\beta$ splits
  completely in $E$.  

  For an irreducible polynomial $A \in R_T$ which has order $q^d-1$
  modulo $M$, by part (3) of Proposition~\ref{prop:PS}, the place $A$
  remains inert in the extension $K/F$, and therefore also in the
  sub-extension $E/F$. Since the degree of the place $A$ equals
  $\ell$, (\ref{eq:prop-of-artin-auto}) follows from the definition of
  the Artin automorphism at $A$.  
\end{proof}

\subsection{A generator for $E$ and its properties}

We would like to represent elements of $E$ and be able to evaluate
them at the places above $T-\beta$. To this end, we will exhibit a
$\mu \in F^{\ac}$ such that $E = F(\mu)$ along with defining equation
for $\mu$ (which will then aid in the evaluations of $\mu$ at the
requisite places).

\begin{theorem}
\label{thm:subfield-structure}
  Let $\lambda$ be an arbitrary nonzero element of $\Lambda_M$ (so
  that $K = F(\lambda)$). Define 
\begin{equation}
\label{eqn:defn-of-mu}
\mu \eqdef \prod_{\sigma \in H}
  \sigma(\lambda) = C_{A^b}(\lambda) C_{A^{2b}}(\lambda) \cdots
  C_{A^{q^d-1}}(\lambda)
\ . 
\end{equation}
 Then, the fixed field $K^H$ equals $E =
  F(\mu)$. The minimal polynomial $h(Z) \in R_T[Z]$ of $\mu$ over $F$
  is given by 
\[ h(Z) = \prod_{j=0}^{b-1} (Z - \Gamma^j(\mu)) \ . \]
 Further, the polynomial $h(Z)$ can be computed in $q^{O(d)}$ time.
\end{theorem}

\begin{proof}
  By definition $\mu$ is fixed by each $\pi \in H$ and so $\mu \in
  E$. Therefore $F(\mu) \subseteq E$.

  To show $E = F(\mu)$, we will argue that $[F(\mu) : F] = b$, which
  in turn follows if we show that $h(Z)$ has coefficients in $F$ and is
  irreducible over $F$. Since $\Gamma^b(\mu) = \mu$ and thus
  $\Gamma^j(\mu)$ only depends on $j \mod b$, all symmetric functions
  of $\{\Gamma^j(\mu)\}_{j=0}^{b-1}$ are fixed by $\Gamma$, and thus also
  by all of $\Gal(K/F)$. The coefficients of $h(Z)$ must therefore belong to $F$. The lemma actually claims that the coefficients lie in $R_T$. To see this, note that for $j=0,1,\dots,b-1$,
\begin{equation}
\label{eq:Gamma-j-mu}
\Gamma^j(\mu) = \prod_{0 \le i < q^d-1 \atop {i \mod
  b = j}} \Gamma^i(\lambda) = \prod_{0 \le i < q^d-1 \atop {i \mod b = j}}
C_{A^i}(\lambda) \ .
\end{equation} 
Since $\lambda$ and all its Galois conjugates $C_{A^i}(\lambda)$ are
integral over $F$, each $\Gamma^j(\mu)$ is integral over $F$, and thus
so is each coefficient of $h(Z)$. But since we already know they
belong to $F$, the coefficients must in fact lie in $R_T$.

We will prove $h(Z)$ is irreducible over $F$ by showing that it is an
Eisenstein polynomial with respect to the place $M$. Since $\mu =
\lambda \times \prod_{\sigma \in H, \sigma \neq 1} \sigma(\lambda)$,
for each $j$, $0 \le j < b$, $\Gamma^j(\mu)$ is divisible by
$\Gamma^j(\lambda)$ in the ring $R_T[\lambda]$. Now $\Gamma^j(\lambda)
= C_{A^j}(\lambda)$ which is divisible by $\lambda$. By
Proposition~\ref{prop:PS}, $\lambda \in \tilde{M}$, and hence each
coefficient of $h(Z)$ belongs to the ideal $F \cap \tilde{M} = M$. (A
reminder that we are using $M$ to denote both the polynomial in $R_T$
and its associated place.) Therefore, all coefficients of $h(Z)$
except the leading coefficient are divisible by $M$.

  The constant term of $h(Z)$ equals
\begin{equation}
\label{eq:mu-is-M-prime}
\prod_{j=0}^{b-1} \Gamma^j(\mu) = \prod_{j=0}^{b-1} \prod_{\sigma
  \in H} \Gamma^j(\sigma(\lambda)) = \prod_{j=0}^{b-1} \prod_{0 \le i
  < (q^d-1)/b} \Gamma^{bi+j}(\lambda) = \prod_{\pi \in G} \pi(\lambda)
= M 
\end{equation}
 where the last step follows since the minimal polynomial of
$\lambda$ over $F$ is $\prod_{\pi \in G} (Z- \pi(\lambda))$, but the
minimal polynomial is also $C_M(Z)/Z$ which has $M$ as the constant
term.  Thus the constant term of $h(Z)$ is not divisible by $M^2$. By
Eisenstein's criterion, we conclude that $h(Z)$ must be irreducible
over $F$.

Finally, we turn to how the coefficients of $h(Z)$ can be computed
efficiently. By the expression (\ref{eq:Gamma-j-mu}), we can compute
$\Gamma^j(\mu)$ for $0 \le j \le b-1$ as a formal polynomial in
$\lambda$ with coefficients from $R_T$. We can divide this polynomial
by the monic polynomial $C_M(\lambda)/\lambda$ (formally, over the
polynomial ring $R_T[\lambda]$) and represent $\Gamma^j(\mu)$ as a
polynomial of degree less than $(q^d-1)$ in $\lambda$. Using this
representation, we can compute the polynomials $h^{(i)}(Z) =
\prod_{j=0}^{i} (Z - \Gamma^j(\mu))$ for $1 \le i \le b-1$
iteratively, as an element of $R_T[\lambda][Z]$, with all coefficients
having degree less than $(q^d-1)$ in $\lambda$. When $i = b-1$, we
would have computed $h(Z)$ --- we know at the end all the coefficients
will have degree $0$ in $\lambda$ and belong to $R_T$.
\end{proof}

By Equation (\ref{eq:mu-is-M-prime}) in the above argument, and the
fact that $v_{M'}(\Gamma^j(\mu)) = v_{M'}(\mu)$, we conclude that
$v_{M'}(\mu) = 1$, i.e. $\mu$ (as well as each of its Galois
conjugates $\Gamma^j(\mu)$) is $M'$-prime. We record this fact
below. It will be used to prove that the integral closure of $R_T$ in
$E$ equals $R_T[\mu]$ (Proposition~\ref{prop:integral-closure}), en route
characterizing the message space in
Theorem~\ref{thm:form-of-messages}.

\begin{lemma}
\label{lem:mu-is-M'-prime}
The element $\mu$ has a simple zero at $M'$, i.e., $v_{M'}(\mu) = 1$.
\end{lemma}

With the minimal polynomial $h(Z)$ of $\mu$ at our disposal, we turn
to computing the evaluations of $\mu$ at the $b$ places above
$T-\beta$, call them $P^{(\beta)}_j$ for $j=0,1,\dots,b-1$, for each
$\beta \in \F_r$. (Recall that the place $T-\beta$ splits completely in $E/F$ by Proposition~\ref{prop:basic-subfield}, Part (v).)
The following lemma identifies the set of evaluations
of $\mu$ at these places.  This method is related to
Kummer's theorem on splitting of primes~\cite[Sec. III.3]{stich-book}.
\begin{lemma}
\label{lem:values-above-beta}
  Consider the polynomial $\bar{h}^{(\beta)}(Z) \in \F_q[Z]$ obtained
  by evaluating the coefficients of $h(Z)$, which are polynomials in
  $T$, at $\beta$. Then $\bar{h}^{(\beta)}(Z) = \prod_{j=0}^{b-1} (Z -
  \mu(P^{(\beta)}_j))$. In particular, the set of evaluations of $\mu$
  at the places above $(T-\beta)$ equals the roots of
  $\bar{h}^{(\beta)}$ in $\F_q$, and can be computed in $b^{O(1)}$
  time given $h \in R_T[Z]$.
\end{lemma}
\begin{proof}
We know $h(Z) = \prod_{j=0}^{b-1} (Z - \Gamma^j(\mu))$. Therefore
\[
\bar{h}^{(\beta)}(Z) = \prod_{j=0}^{b-1} (Z - \Gamma^j(\mu)(P^{(\beta)}_0)) 
=  \prod_{j=0}^{b-1} \Bigl(Z - \mu\bigl(\Gamma^{-j}(P^{(\beta)}_0)\bigr) \Bigr)  = \prod_{j=0}^{b-1} (Z - \mu(P^{(\beta)}_j)) \]
where the last step uses the fact that $\Gamma^{-j}(P^{(\beta)}_0)$ for $j=0,1,\dots,b-1$ is precisely the set of places above $T-\beta$. 
\end{proof}

\section{Code construction from cyclotomic function field}
\label{sec:code-const}

We will now describe the algebraic-geometric codes based on the
function field $E$. A tempting choice for the message space is perhaps
$\{\sum_{i=0}^{b-1} a_i(T) \mu^i\} \subset R_T[\mu]$ where $a_i(T)$
are polynomials of some bounded degree. This is certainly a $\F_q$-linear space and messages in this space have no poles
outside the places lying above $P_\infty$. However, the valuations of
$\mu$ at these places is complicated (one needs the Newton polygon
method to estimate these~\cite[Sec. 12.4]{salvador}), and since $\mu$
has both zeroes and poles amongst these places, it is hard to get good
bounds on the total pole order of such messages at each of the places above $P_\infty$.

\subsection{Message space}
Let $M'$ be the unique totally ramified place $M'$ in $E$ lying above
$M$; $\deg(M') =\deg(M) = d$.  We will use as message space elements
of $R_T[\mu]$ that have no more than a certain number $\ell$ of poles
at the place $M'$ and no poles elsewhere. These can equivalently be
thought of (via a natural correspondence) as elements of $E$ that have
bounded (depending on $\ell$) pole order at each place above
$P_\infty$, and no poles elsewhere, and we can develop our codes and
algorithms in this equivalent setting. Since the literature on AG
codes typically focuses on one-point codes where the messages have
poles at a unique place, we work with functions with poles restricted
to $M'$.

Formally, for an integer $\ell \ge 1$, let $\L(\ell M')$ be the space
of functions in $E$ that have no poles outside $M'$ and at most $\ell$
poles at $M'$. $\L(\ell M')$ is an $\F_q$-vector space, and by the
Riemann-Roch theorem, $\dim (\L(\ell M')) \ge \ell d - g + 1$, where
$g = d(b-1)/2+1$ is the genus of $E$.  We will assume that $\ell \ge
b$, in which case $\dim(\L(\ell M')) = \ell d - g +1$. 

We will represent the code by a basis of $\L(\ell M')$ over $\F_q$. Of
course, we first need to understand how to represent a single function
in $\L(\ell M')$.  The following lemma suggest a representation for
elements of $\L(\ell M')$ that we can use.

% appendix pointer
% The proof is presented in Appendix~\ref{app:pfs-omit-4}.

\begin{theorem}
\label{thm:form-of-messages}
A function $f$ in $E$ with poles only at $M'$ has a unique
representation of the form
\begin{equation}
\label{eq:form-of-basis-fns}
f = \frac{\sum_{i=0}^{b-1} a_i \mu^i}{M^e} 
\end{equation} 
where $e \ge 0$ is an
integer, each $a_i \in R_T$, and not all the $a_i$'s are divisible by
$M$ (as polynomials in $T$).
\end{theorem}
\begin{proof}
  If $f$ has poles only at $M'$, there must be a smallest integer $e
  \ge 0$ such that $M^e f$ has no poles outside the places above
  $P_\infty$. This means that $M^e f$ must be in the integral closure (``ring of integers'')
  of $R_T$ in $E$, i.e., the minimal polynomial of $M^e f$ over $R_T$
  is monic. The claim will follow once we establish that the integral
  closure of $R_T$ in $E$ equals $R_T[\mu]$, which we show next in
  Proposition~\ref{prop:integral-closure}. The uniqueness follows since $\{1,\mu,\dots,\mu^{b-1}\}$ forms a basis of $E$ over $F$.
\end{proof}

\begin{prop}
\label{prop:integral-closure}
The integral closure of $R_T$ in $E$ equals $R_T[\mu] = \Bigl\{ \sum_{i=0}^{b-1} a_i \mu^i ~ \mid ~ a_i \in R_T\Bigr\}$.
\end{prop}
\begin{proof}
  The minimal polynomial $h(Z)$ of $\mu$ over $R_T$ is monic
  (Theorem~\ref{thm:subfield-structure}). Thus $\mu$ is integral over
  $R_T$, and so $R_T[\mu]$ is contained in the integral closure of
  $R_T$ in $E$. We turn to proving the reverse inclusion. The proof
  follows along the lines of a similar argument used to prove that the
  integral closure of $R_T$ in $K = F(\lambda)$ equals
  $R_T[\lambda]$~\cite[Prop. 12.9]{rosen}. Let $\omega \in E$ be
  integral over $R_T$.  We know that $\{1,\mu,\mu^2,\dots,\mu^{b-1}\}$
  is a basis for $E$ over $F$. Also $\mu$, and therefore each $\mu^i$,
  is integral over $F$. By virtue of these facts, it is known (see,
  for example, \cite[Chap. 2]{marcus}) that there exist $a_i \in R_T$
  such that $\omega = \frac{1}{\Delta} \sum_{i=0}^{b-1} a_i \mu^i$
  where $\Delta \in R_T$ is the {\em discriminant} of the extension
  $E/F$. As $M$ is the only ramified place in the extension $E/F$, the
  discriminant $\Delta$ is a power of $M$ up to units, and by assuming
  wlog that $\Delta$ is monic, we can conclude that $\Delta = M^{e'}$
  for some exponent $e' \ge 0$.  Thus we have
\begin{equation}
\label{eq:integ-closure}
M^{e'} \omega = \sum_{i=0}^{b-1} a_i \mu^i 
\end{equation}
with $a_i \in R_T$, and not all the $a_i$'s are divisible by $M$.

Our goal is to show that $e' = 0$. We will do this by comparing the valuations $v_{M'}$ of the both sides of (\ref{eq:integ-closure}). We have
\begin{equation}
\label{eq:cwq09}
v_{M'}(M^{e'} \omega) = v_{M'}(M^{e'}) + v_M(\omega) = b e' + v_M(\omega) \ge b e' \ . 
\end{equation}
Let $i_0$, $0 \le i_0 < b$, be the smallest value of $i$ such that $v_{M}(a_i) = 0$. Such an $i_0$ must exist since not all the $a_i$'s are divisible by $M$.
By Lemma~\ref{lem:mu-is-M'-prime}, $v_{M'}(\mu) = 1$, and so 
\[ v_{M'}(a_i \mu^i) = v_{M'}(a_i) + i = b v_M(a_i) + i \ . \] 
For $i=i_0$, $v_{M'}(a_{i_0} \mu^{i_0}) = i_0$. For $i < i_0$, $v_{M'}(a_i \mu^i) \ge b v_{M}(a_i) \ge b > i_0$ (since $v_{M}(a_i) \ge 1$ for $i < i_0$). For $i > i_0$, $v_{M'}(a_i \mu^i) \ge v_{M'}(\mu^i) = i > i_0$. It follows that
\begin{equation}
\label{eq:cwq10}
v_{M'}\Bigl(\sum_{i=0}^{b-1} a_i \mu^i\Bigr) = \min_{0 \le i \le b-1} v_{M'}(a_i \mu^i) =  i_0
\end{equation}
Combining (\ref{eq:cwq09}) and (\ref{eq:cwq10}), we conclude $b > i_0 \ge b e'$ which implies $e'=0$.
\end{proof}

\subsection{Succinctness of representation}
In order to be able to efficiently compute with the representation
(\ref{eq:form-of-basis-fns}) of functions in $\L(\ell M')$, we need
the guarantee that the representation will be {\em succinct}, i.e., of
size polynomial in the code length. We show that this will be the case
by obtaining an upper bound on the degree of the coefficients $a_i \in
R_T$ in Lemma~\ref{lem:ub-on-degree} below. This is not as
straightforward as one might hope, and we thank G. Anderson and
D. Thakur for help with its proof. For the choice of parameters we
will make (in Theorems \ref{thm:main-cyclo} and \ref{thm:main-final}), this upper bound will be
polynomially bounded in the code length. Therefore, the assumed
representation of the basis functions is of polynomial size.
\begin{lemma}
\label{lem:ub-on-degree}
Suppose $f \in \L(\ell M')$ is given by $f = \frac{1}{M^e}
\sum_{i=0}^{b-1} a_i \mu^i$ for $a_i \in R_T$ (not all divisible by
$M$) and $e \ge 0$. Then the degree of each $a_i$ is at most $\ell +
q^d b$.
\end{lemma}
\begin{proof}
  Let $g = M^e f = \sum_{i=0}^{b-1}a_i\mu^i$. We know that $g$ has at
  most $eb$ poles at each place of $E$ that lies above
  $P_\infty$ (since $f$ has no poles at these places). Using the fact that $f$ has at most $\ell$ poles at $M'$, and the uniqueness of the representation 
 $f = \frac{1}{M^e}
\sum_{i=0}^{b-1} a_i \mu^i$, it is easy to argue that $eb \le \ell +b$. So, $g$ has at most $\ell + b$ poles at each place of $E$ lying above $P_\infty$.

Let $\sigma = \sigma_A$; we know that $\sigma$ is a generator of
$\Gal(E/F)$. For $j=0,1,\dots,b-1$, we have $\sigma^j(g) =
\sum_{i=0}^{b-1} a_i \sigma^j(\mu^i)$.  Let $\mv{a} =
(a_0,a_1,\dots,a_{b-1})^T$ be the (column) vector of coefficients, and
let $\mv{g} = (g,\sigma(g),\dots,\sigma^{b-1}(g))^T$. Denoting by
$\Phi$ the $b \times b$ matrix with $\Phi_{ji} = \sigma^j(\mu^i)$ for
$0 \le i,j \le b-1$, we have the system of equations $\Phi \mv{a} =
\mv{b}$.

  We can thus determine the coefficients $a_i$ by solving this linear
  system. By Cramer's rule, $a_i = {\rm det}(\Phi_i)/{\rm det}(\Phi)$
  where $\Phi_i$ is obtained by replacing the $i$'th column of $\Phi$
  by the column vector $\mv{g}$.  The square of the denominator ${\rm
    det}(\Phi)$ is the discriminant of the field extension $E/F$, and
  belongs to $R_T$. Thus the degree of $a_i$ is at most the pole order
  of ${\rm det}(\Phi_i)$ at an arbitrary place, say $\tilde{P}$, above
  $P_\infty$.  By the definition (\ref{eqn:defn-of-mu}) of $\mu$, and
  the fact that $\lambda$ and its conjugates have at most one pole at
  the places above $P_\infty$ in $F(\Lambda_M)$, it follows that $\mu$
  has at most $(q^d-1)/b$ poles at $\tilde{P}$. The same holds for all
  its conjugates $\sigma^j(\mu)$. The function $g$ and its conjugates
  $\sigma^j(g)$ have at most $\ell+b$ poles at $\tilde{P}$.  In all,
  this yields a crude upper bound of 
  \[ \frac{q^d-1}{b} \frac{(b-1) b}{2} + \ell + b \le \ell + q^d b \]
  for the pole order of ${\rm det}(\Phi_i)$ at $\tilde{P}$, and hence
  also the degree of the polynomial $a_i \in R_T$.
\end{proof}

\subsection{Rational places for encoding and their ordering}
\label{sec:hqp}
So far, the polynomial $A \in R_T$ was any monic irreducible
polynomial that was a primitive element modulo $M$, so that its Artin
automorphism $\sigma_A$ generates $\Gal(E/F)$.  We will now pick
$A$ to have degree $D$ satisfying $D > \frac{\ell d}{b}$. This can be
done by a Las Vegas algorithm in $(D q^d)^{O(1)}$ time by picking a random
polynomial and checking that it works, or deterministically by brute
force in $q^{O(d+D)}$ time. Either of these lies within the decoding
time claimed in Theorem~\ref{thm:main-cyclo}, and will be polynomial
in the block length for our parameter choices in
Theorem~\ref{thm:main-final}. By Proposition~\ref{prop:basic-cyclo},
$A$ remains inert in $E/F$, and let us denote by $A'$ the unique place
of $E$ that lies over $A$. The degree of $A'$ equals $D b$.

For each $\beta \in \F_r$, fix an arbitrary place $P^{(\beta)}_0$
lying above $T-\beta$ in $E$. For $j=0,1,\dots,b-1$, define
\begin{equation}
\label{eq:place-ordering}
P^{(\beta)}_j = \sigma_A^{-j}(P^{(\beta)}_0)  \ .
\end{equation}
Since $\Gal(E/F)$ acts
transitively on the set of primes above a prime, and $\sigma_A$
generates $\Gal(E/F)$, these constitute all the places above $T-\beta$.
Lemma~\ref{lem:values-above-beta} already tells us the {\em set} of
evaluations of $\mu$ at these places, but not which evaluation
corresponds to which point. We have $\mu(\sigma_A^{-j}(P^{(\beta)}_0))
= \sigma_A^j(\mu)(P^{(\beta)}_0)$; hence, to compute the evaluations
of $\mu$ at all these $b$ places as per the ordering
(\ref{eq:place-ordering}), it suffices to know
\begin{enumerate}
\itemsep=1ex
\item the value at $\mu(P^{(\beta)}_0)$, which we can find by
  simply picking one one of the roots from
  Lemma~\ref{lem:values-above-beta} arbitrarily, and
\item a representation of
$\sigma_A(\mu)$ as an element of $R_T[\mu]$ (since $\sigma_A(\mu)$ is
integral over $R_T$, it belongs to $R_T[\mu]$ by virtue of
Proposition~\ref{prop:integral-closure}). Note that $T(P^{(\beta)}_0) =
\beta$, so once we know $\mu(P^{(\beta)}_0)$, we can evaluate any
element of $R_T[\mu]$ at $P^{(\beta)}_0$.  
\end{enumerate} 

\smallskip \noindent 
We now show that $\sigma_A(\mu) \in R_T[\mu]$ can be computed efficiently.
\begin{lemma}
\label{lem:compute-mu-vals}
\begin{enumerate}
\itemsep=1ex
\item
The values of $\sigma_A^j(\mu)$ for $0 \le j \le b-1$ as elements of $R_T[\mu]$ can be computed in $q^{O(d)}$ time. 
\item The values $\mu(P^{(\beta)}_j)$ for $\beta \in \F_r$ and
  $j=0,1,\dots,b-1$ can be computed in $q^{O(d)}$ time. Knowing these
  values, we can compute any function in the message space $\L(\ell
  M')$ represented in the form (\ref{eq:form-of-basis-fns}) at the
  places $P^{(\beta)}_j$ in ${\rm poly}(\ell, q^d)$ time.
\end{enumerate}
\end{lemma}
\begin{proof}
Part (ii) follows from Part (i) and the discussion above. To prove Part (i), note that once we compute $\sigma_A(\mu)$, we can recursively compute
  $\sigma^j_{A}(\mu)$ for $j \ge 2$, using the relation $h(\mu) = 0$
  to replace $\mu^b$ and higher powers of $\mu$ in terms of
  $1,\mu,\dots,\mu^{b-1}$. By definition (\ref{eqn:defn-of-mu}), we
  have $\mu = \prod_{0 \le i < (q^d-1)/b} C_{A^{ib} \mod
    M}(\lambda)$. Thus one can compute an expression $\mu =
  \sum_{i=0}^{q^d-2} e_i \lambda^i \in R_T[\lambda]$ with coefficients
  $e_i \in R_T$ in $q^{O(d)}$ time. By successive multiplication in the
  ring $R_T[\lambda]$ (using the relation $C_M(\lambda) = 0$ to
  express $\lambda^{q^d-1}$ and higher powers in terms of
  $1,\lambda,\dots,\lambda^{q^d-2}$), we can compute, for $l
  =0,1,\dots,b-1$, expressions $\mu^l = \sum_{i=0}^{q^d-2} e_{il}
  \lambda^i$ with $e_{il} \in R_T$ in $q^{O(d)}$ time.

  We have $\sigma_A(\mu) = \sum_{i=0}^{q^d-2} e_i \sigma_A(\lambda)^i
  = \sum_{i=0}^{q^d-2} e_i C_{A \mod M}(\lambda)^i$. So one can
  likewise compute an expression $\sigma_A(\mu) = \sum_{i=0}^{q^d-2}
  f_i \lambda^i$ with $f_i \in R_T$ in $q^{O(d)}$ time. The task now
  is to re-express this expression for $\sigma_A(\mu)$ as an element
  of $R_T[\mu]$, of the form $\sum_{l=0}^{b-1} a_l \mu^l$, for
  ``unknowns'' $a_l \in R_T$ that are to be determined. We will argue
  that this can be accomplished by solving a linear system.

  Indeed, using the above expressions $\mu^l = \sum_{i=0}^{q^d-2}
  e_{il} \lambda^i$, the coefficients $a_l$ satisfy the following system of
  linear equations over $R_T$:
\begin{equation}
\label{eq:sys-over-R_T}
\sum_{l=0}^{b-1} e_{il} a_l = f_i \quad \mbox{for} \quad
  i=0,1,\dots,q^d-2  \ .
\end{equation}
Since the representation $\sigma_A(\mu) = \sum_{l=0}^{b-1} a_l \mu^l$
is unique, the system has a unique solution. By Cramer's rule, the
degree of each $a_l$ is at most $q^{O(d)}$. Therefore, we can express
the system (\ref{eq:sys-over-R_T}) as a linear system of size $q^{O(d)}$ over $\F_q$ in
unknowns the coefficients of all the polynomials $a_l \in R_T$. By solving this system in $q^{O(d)}$ time, we can compute the representation of $\sigma_A(\mu)$ as an element of $R_T[\mu]$. 
\end{proof}

\subsection{The basic cyclotomic AG code}
The basic AG code $\C^0$ based on subfield $E$ of the cyclotomic function field
$F(\Lambda_M)$ is defined as
\begin{equation}
\label{eq:basic-cycl}
\C^0 = \left\{ \Bigl( f(P^{(\beta)}_j) \Bigr)_{\beta \in F_r, 0 \le j <
    b}  ~ \mid ~ f \in \L(\ell M') \right\}
\end{equation}
where the ordering of the places $P^{(\beta)}_j$ above $T-\beta$ is
as in (\ref{eq:place-ordering}).
We record the standard parameters of the above algebraic-geometric
code, which follows from Riemann-Roch, the genus of $E$ from
Proposition~\ref{prop:basic-subfield}, and the fact a nonzero $f \in \L(\ell
M')$ can have at most $\ell \cdot \deg(M') = \ell d$ zeroes.
\begin{lemma}
\label{lem:params-of-unfolded-code}
Let $\ell \ge b$. $\C^0$ is an $\F_q$-linear code of block length $n=rb$, dimension $k = \ell d - d(b-1)/2$, and distance at least $n - \ell d$. 
\end{lemma}
Lemma~\ref{lem:compute-mu-vals}, Part (ii), implies the following.
\begin{lemma}[Efficient encoding]
\label{lem:gen-matrix-comp}
  Given a basis for the message space $\L(\ell M')$ represented in the
  form (\ref{eq:form-of-basis-fns}), the generator matrix of the
  cyclotomic code $\C^0$ can be computed in ${\rm poly}(\ell, q^d, q^D)$
  time.
\end{lemma}

\subsection{The folded cyclotomic code}
Let $m \ge 1$ be an integer. For convenience, we assume $m | b$
(though this is not really necessary). Analogous to the construction
of folded Reed-Solomon codes~\cite{GR-capacity}, the folded cyclotomic code $\C$
is obtained from $\C^0$ by bundling together successive $m$-tuples of
symbols into a single symbol to give a code of length $N = n/m$ over
$\F_q^m$. Formally,
\begin{equation}
\label{eq:code-def}
\C = \left\{ \Bigl( f(P^{(\beta)}_{m\imath}), f(P^{(\beta)}_{m\imath+1}), \cdots
  , f(P^{(\beta)}_{m\imath+m-1}) \Bigr)_{\beta \in F_r, 0 \le \imath <
    b/m}  ~ \mid ~ f \in \L(\ell M') \right\}
\end{equation}
We will index the $N$ positions of codewords in $\C$ by  pairs $(\beta,\imath)$ for $\beta \in \F_r$ and $\imath \in \{0,1,\dots,\frac{b}{m}-1\}$.

The generator matrix of unfolded code $\C^0$, which can be computed
given a basis for $\L(\ell M')$ as per
Lemma~\ref{lem:gen-matrix-comp}, obviously suffices for encoding. We
will later on argue that the {\em same representation} also suffices for
polynomial time list decoding.

\subsection{Folding and Artin-Frobenius automorphism}
The unique place $A'$ lying above $A$ has degree $D' \eqdef Db$. The
residue field at $A'$, denote it $K_{A'}$, is isomorphic to
$\F_{q^{D'}}$. By our choice $Db > \ell d$. This immediately implies a
message in $\L(\ell M')$ is uniquely determined by its evaluation at
$A'$.
\begin{lemma}
\label{lem:ev-one-one}
  The map $\ev_{A'} : \L(\ell M') \rightarrow K_{A'}$ given by
  $\ev_{A'}(f) = f(A')$ is one-one.
\end{lemma}
The key algebraic property of our folding is the following.
\begin{lemma}
\label{lem:key-folding}
For every $f \in \L(\ell M')$:
\begin{enumerate}
\item For every $\beta \in \F_r$ and $0 \le j < b-1$,
  $\sigma_A(f)(P^{(\beta)}_j) = f(P^{(\beta)}_{j+1})$.
\item $\sigma_A(f)(A') = f(A')^{q^D}$.
\end{enumerate}
\end{lemma}
\begin{proof}
The first part follows since we ordered the places above
$T -\beta$ such that $P^{(\beta)}_{j+1} = \sigma_A^{-1} (P^{(\beta)}_j)$.

The second part follows from the property of the Artin automorphism at
$A$, since the norm of the place $A$ equals $q^{\deg(A)} = q^D$. (A
nice discussion of the Artin-Frobenius automorphism, albeit in the
setting of number fields, appears in \cite[Chap. 4]{marcus}.)
\end{proof}

\section{List decoding algorithm}
\label{sec:ld-cyclo}
We now turn to list decoding the folded cyclotomic code $\C$ defined in
(\ref{eq:code-def}). The underlying approach is similar to that of the
algorithm for list decoding folded RS codes~\cite{GR-capacity} and
algebraic-geometric generalizations of Parvaresh-Vardy
codes~\cite{PV-focs05,GP-mathcomp}.  We will therefore not repeat the
entire rationale and motivation behind the algorithm development.  But
our technical presentation and analysis is self-contained.  In fact,
our presentation here does offer some simplifications over previous
descriptions of AG list decoding algorithms from
\cite{GS,GS01-rep,GP-mathcomp}. A principal strength of the new
description is that it {\sl avoids the use of zero-increasing bases}
at each code place $P^{(\beta)}_j$. This simplifies the algorithm as
well as the representation of the code needed for decoding.

The list decoding problem for $\C$ up to $e$ errors corresponds to
solving the following function reconstruction problem. Recall that the
length of the code is $N = n/m = r b/m$, and the codeword positions are indexed by $\F_r \times \{0,1,\dots,\frac{b}{m}-1\}$.
\begin{description}
\itemsep=1ex
\item[Input] Collection ${\mathcal T}$ of $N$ tuples 
$\Bigl( y^{(\b)}_{m\imath}, y^{(\b)}_{m\imath+1},\cdots, y^{(\b)}_{m\imath+m-1} \Bigr) \in \F_q^m$ for $\b \in \F_r$ and $0 \le \imath < b/m$
\item[Output] A list of all $f \in \L(\ell M')$ whose encoding as per
  $\C$ agrees with the $(\b,\imath)$'th tuple for at least $N-e$ codeword
  positions.
\end{description}

\subsection{Algorithm description}

We describe the algorithm at a high level below and later justify how
the individual steps can be implemented efficiently, and under what
condition the decoding will succeed. We stress that regardless of
complexity considerations, even the {\em combinatorial}
list-decodability property ``proved'' by the algorithm is non-trivial.

\smallskip
\noindent {\sf Algorithm List-Decode($\C$):}  ~~(uses the following parameters):
\begin{itemize}
\itemsep=1ex
\item an integer parameter $s$, $2 \le s \le m$, for $s$-variate interpolation
\item an integer parameter $w \ge 1$ that governs the zero order (multiplicity) guaranteed by interpolation
\item an integer parameter $\Delta \ge 1$ which is the total degree of the interpolated $s$-variate polynomial
\end{itemize}
\begin{description}
\itemsep=1ex
\item[Step 1] (Interpolation) Find a nonzero polynomial
  $Q(Z_1,Z_2,\dots,Z_s)$ of total degree at most $\Delta$ with
  coefficients in $\L(\ell M')$ such that for each $\b \in \F_r$, $0
  \le \imath < b/m$, and $j' \in \{0,1,\dots,m-s\}$, the shifted polynomial
\begin{equation}
\label{eq:shift-poly}
 Q \bigl( Z_1 + y^{(\b)}_{m\imath+j'}, Z_2 + y^{(\b)}_{m\imath+j'+1}, \cdots , Z_s + y^{(\b)}_{m\imath+j'+s-1} \bigr) 
\end{equation}
has the property that the coefficient of the monomial $Z_i^{n_1} Z_2^{n_2} \cdots Z_s^{n_s}$ vanishes at $P^{(\b)}_{m\imath+j'}$ whenever its total degree $n_1+n_2+\cdots+n_s < w$.
\item[Step 2] (Root-finding) Find a list of all $f \in \L(\ell M')$ satisfying
  \[ Q(f,\sigma_A(f),\dots,\sigma_{A^{s-1}}(f)) = 0 \ . \] Output those whose
  encoding as per the code $\C$ agrees with at least $N-e$ of the
  $m$-tuples in ${\mathcal T}$.
\end{description}

\subsection{Analysis of error-correction radius}

\begin{lemma}
\label{lem:interpolation}
If $k (\Delta+1)^s \ge N (m-s+1) (w+s-1)^s$ (where, recall, $k = \ell
d - d(b-1)/2$ is the dimension of $\L(\ell M')$), then a nonzero
polynomial $Q$ with the stated properties exists. If we know the
evaluations of the functions in a basis
$\{\phi_1,\phi_2,\dots,\phi_k\}$ of $\L(\ell M')$ at the places
$P^{(\b)}_j$, then such a $Q$ can be found by solving a homogeneous
system of linear equations over $\F_q$ with at most $Nm (w+s)^s$
equations and unknowns.
\end{lemma}
\begin{proof}
  The proof is standard and follows by counting degrees of freedom
  vs. number of constraints. One can express the desired polynomial as
  $\sum_{n_1,n_2,\dots,n_s} q_{(n_1,\dots,n_s)} Z_1^{n_1} \cdots
  Z_s^{n_s}$ with unknowns $q_{(n_1,\dots,n_s)} \in \F_q$. The number
  of coefficients is $k {{\Delta+s} \choose s} > k (\Delta+1)^s/
  s!$. For each place $P^{(\b)}_{m\imath+j'}$, one can express the required
  condition at that place by ${{w+s-1} \choose s}$ linear conditions
  (this quantity is the number of monomials of total degree $< w$),
  for a total of \[ N (m-s+1) {{w+s-1} \choose s} < N(m-s+1)
  \frac{(w+s-1)^s}{s!} \] constraints. When the number of unknowns exceeds the
  number of constraints, a nonzero solution must exist. A solution can
  also be found efficiently once the linear system is set up, which
  can clearly be done if we know the evaluations of $\phi_i$'s at the
  code places (i.e., a ``generator matrix'' of the code).
\end{proof}

\begin{lemma}
\label{lem:zero-count-1}
  Let $Q$ be the polynomial found in Step 1. If the encoding of some
  $f$ as per $\C$ agrees with $\bigl( y^{(\b)}_{m\imath},
  y^{(\b)}_{m\imath+1},\cdots, y^{(\b)}_{m\imath+m-1} \bigr)$ for some position
  $(\b,\imath)$, then $Q(f,\sigma_A(f),\dots,\sigma_{A^{s-1}}(f))$ has at
  least $w$ zeroes at each of the $(m-s+1)$ places $P^{(\b)}_{m\imath+j'}$
  for $j'=0,1,\dots,m-s$.
\end{lemma}

\begin{proof}
  The proof differs slightly from earlier proofs of similar statements
  (eg., \cite[Lemma 6.6]{GP-mathcomp}) in that it avoids the use of
  zero-increasing bases and is thus simpler. We will prove the claim
  for $j'=0$, and the same proof works for any $j' \le m-s$. Note that
  agreement on the $m$-tuple at position $(b,\imath)$ implies that
\[ f(P^{(\b)}_{m\imath}) = y^{(\b)}_{m\imath}, ~~~ f(P^{(\b)}_{m\imath+1}) = y^{(\b)}_{m\imath+1}, ~~~ \cdots, ~~~ f(P^{(\b)}_{m\imath+s-1}) = y^{(\b)}_{m\imath+s-1} \ . \]
By Lemma~\ref{lem:key-folding}, Part (i), this implies
\[ f(P^{(\b)}_{m\imath}) = y^{(\b)}_{m\imath}, ~~~ \sigma_A(f)(P^{(\b)}_{m\imath}) = y^{(\b)}_{m\imath+1}, ~~~ \cdots, ~~~ \sigma_{A^{s-1}}(f)(P^{(\b)}_{m\imath}) = y^{(\b)}_{m\imath+s-1}  \ .\]
Denote by $Q^*$ the shifted polynomial (\ref{eq:shift-poly}) for the triple $(\b,\imath,0)$. We have
\begin{align*}
& Q\bigl(f,\sigma_A(f),\dots,\sigma_{A^{s-1}}(f)\bigr) =~  Q^*\bigl(f -  y^{(\b)}_{m\imath}, \sigma_A(f) -  y^{(\b)}_{m\imath+1} , \cdots,  \sigma_A^{s-1}(f)- y^{(\b)}_{m\imath+s-1}\bigr) \\
= & ~ \sum_{{n_1,n_2,\dots,n_s} \atop {w \le n_1+\cdots+n_s \le \Delta}} \hspace*{-5mm} q^{*}_{(n_1,\dots,n_s)} \bigl( f -  f(P^{(\b)}_{m\imath}) \bigr)^{n_1}  \bigl( \sigma_A(f) - \sigma_A(f)(P^{(\b)}_{m\imath}) \bigr)^{n_2} \cdots \bigl( \sigma_{A^{s-1}}(f) - \sigma_{A^{s-1}}(f)(P^{(\b)}_{m\imath}) \bigr)^{n_s} \ .
\end{align*}
for some coefficients $q^*_{(n_1,\dots,n_s)} \in \F_q$.
Each term of the function in the last expression clearly has valuation
at least $w$ at $P^{(\b)}_{m\imath}$, and hence so does
$Q\bigl(f,\sigma_A(f),\dots,\sigma_{A^{s-1}}(f)\bigr)$.
\end{proof}

\iffalse
%% appendix pointer
\begin{proof}
The proof differs from earlier proofs of similar statements in \cite{GS,GP-mathcomp} in that it avoids the use of zero-increasing bases. Due to space restrictions, the proof is deferred to Appendix~\ref{app:pfs-omit-5}.
\end{proof}
\fi

\begin{lemma}
\label{lem:decoding-cond}
  If the encoding of $f \in \L(\ell M')$ has at least $N-e$ agreements
  with the input tuples ${\mathcal T}$, and $(N-e) (m-s+1) w > d \ell
  (\Delta+1)$, then $Q(f,\sigma_A(f),\dots,\sigma_{A^{s-1}}(f)) = 0$.
\end{lemma}
\begin{proof}
  Since $f$ has no poles outside $M'$, neither do $\sigma_{A^i}(f)$
  for $1 \le i < s$. Moreover, $v_{M'}(\sigma_A(f)) =
  v_{\sigma_A^{-1}(M')}(f) = v_{M'}(f)$ (since $M'$ is the unique
  place above $M$ and is thus fixed by every Galois
  automorphism). Since $f \in \L(\ell M')$, this implies
  $\sigma_{A^i}(f) \in \L(\ell M')$ for every $i$. Since each
  coefficient of $Q$ also belongs to $\L(\ell M')$, we conclude that $
  Q(f,\sigma_A(f),\dots,\sigma_{A^{s-1}}(f)) \in \L((\ell + \ell
  \Delta) M')$. On the other hand, by Lemma~\ref{lem:zero-count-1},
  $Q(f,\sigma_A(f),\dots,\sigma_{A^{s-1}}(f))$ has at least $(N-e)
  (m-s+1) w$ zeroes. If $(N-e) (m-s+1) w > \ell (\Delta+1) d$, then
  $Q(f,\sigma_A(f),\dots,\sigma_{A^{s-1}}(f))$ has more zeroes than
  poles and must thus equal $0$.
\end{proof}

Putting together the above lemmas, we can conclude the following about
the list decoding radius guaranteed by the algorithm. Note that we
have not yet discussed how Step 2 may be implemented, or why it
implies a reasonable bound on the output list size. We will do this in
Section~\ref{sec:step2-impl}.
\begin{theorem}
\label{thm:bd-on-errs}
For every $s$, $2 \le s \le m$, and any $\zeta > 0$, for the choice $w
= \lceil s/\zeta \rceil$ and a suitable choice of the parameter
$\Delta$, the algorithm {\sf List-Decode($\C$)} successfully list
decodes up to $e$ errors whenever
\begin{equation}
\label{eq:bd-on-errors}
e < (N - 1) - (1+\zeta) \left( \frac{k}{m-s+1} \right)^{1-1/s}  N^{1/s} \left( 1 + \frac{d(b-1)}{2k} \right) \ .
\end{equation}
\end{theorem}
\begin{proof}
  Picking $w = \lceil s/\zeta \rceil$ and $\Delta + 1 = \left\lceil
    \left( \frac{N(m-s+1)}{k} \right)^{1/s} (w+s-1) \right\rceil$, the
  requirement of Lemma~\ref{lem:interpolation} is met. By
  Lemma~\ref{lem:params-of-unfolded-code}, the dimension $k$ satisfies
  $\ell d = k + d(b-1)/2$. A straightforward computation reveals that
  for this choice, the bound (\ref{eq:bd-on-errors}) implies the
  decoding condition $(N-e)(m-s+1) w > \ell d (\Delta +1)$ under which
  Lemma~\ref{lem:decoding-cond} guarantees successful decoding.
\end{proof}

\iffalse
%% appendix pointer
%
\begin{remark}
{%\small
  The above error-correction radius is non-trivial only when $s \ge
  2$.  For AG codes, even $s=1$ led to a non-trivial guarantee of
  about $1-\sqrt{R}$ in \cite{GS}. The weaker bound we get is due to
  restricting the pole order of coefficients of $Q$ to at most $\ell$.
  Since we let grow $s$ anyway, this does not hurt us. On the positive
  side, it avoids some difficult technical complications that would
  arise otherwise, and allows implementing the interpolation step just
  using the natural generator matrix of the code. See
  \cite{GP-mathcomp} for a related discussion.}
%This is similar to the  algorithm in \cite[Sec. 5]{GP-mathcomp}.
\end{remark}
\fi

\begin{remark}
  The above error-correction radius is non-trivial only when $s \ge
  2$. We will see later how to pick parameters so that the error
  fraction approaches $1 - R^{1-1/s}$.  For AG codes, even $s=1$ led
  to a non-trivial guarantee of about $1-\sqrt{R}$ in \cite{GS}, and
  for folded Reed-Solomon codes the error fraction with $s$-variate
  interpolation was $1-R^{s/(s+1)}$. The weaker bound we get is due to
  restricting the pole order of coefficients of $Q$ to at most $\ell$,
  the number of poles allowed for messages. This is similar to the
  algorithm in \cite[Sec. 5]{GP-mathcomp}. Since we let grow $s$
  anyway, this does not hurt us. It also avoids some difficult
  technical complications that would arise otherwise (discussed,
  eg. in \cite{GP-mathcomp}), and allows implementing the
  interpolation step just using the natural generator matrix of the
  code.
\end{remark}

\subsection{Root-finding using the Artin automorphism}
\label{sec:step2-impl}
So far we have not discussed how Step 2 of decoding can be performed,
and why in particular it implies a reasonably small upper bound on the
number of solutions $f \in \L(\ell M')$ that it may find in the
worst-case. We address this now. This is where the properties of the
Artin automorphism $\sigma_A$ will play a crucial role. Recall (i)
$K_{A'} = \O_{A'}/A'$ denotes the residue field at the place $A'$ of
$E$ lying above $A$, and (ii) we picked $A$ so that $D = \deg(A)$
obeyed $D b > \ell d$.

\begin{lemma}
  Suppose $f \in \O_{A'}$ satisfies
  \[ Q(f,\sigma_A(f),\dots,\sigma_{A^{s-1}}(f)) = 0 \] for some 
  $Q \in \O_{A'}[Z_1,Z_2,\dots,Z_s]$. Let $\overline{Q} \in
  K_{A'}[Z_1,Z_2,\dots,Z_s]$ be the polynomial obtained by reducing
  the coefficients of $Q$ modulo $A'$. Then $f(A') \in K_{A'}$ obeys
\begin{equation}
\label{eq:artin-root-1}
\overline{Q}\bigl( f(A'), f(A')^{q^D}, f(A')^{q^{2D}}, \cdots, f(A')^{q^{D(s-1)}} \bigr) = 0 \ . 
\end{equation}
\end{lemma}
\begin{proof}
  If $Q(f,\sigma_A(f),\dots,\sigma_{A^{s-1}}(f)) = 0$, then surely
  $\overline{Q}\bigl( f(A') , \sigma_A(f)(A'), \cdots,
  \sigma_{A^{s-1}}(f)(A') \bigr) = 0$. The claim
  (\ref{eq:artin-root-1}) now follows immediately from
  Lemma~\ref{lem:key-folding}, Part (ii).
\end{proof}
\begin{lemma}
\label{lem:roots-Phi}
  If $Q(Z_1,\dots,Z_s)$ is a {\em nonzero} polynomial of total degree
  at most $\Delta < q^D$ all of whose coefficients belong to $\L(\ell
  M')$, then the polynomial $\Phi \in K_{A'}[Y]$ defined as 
\[ \Phi(Y) \eqdef \overline{Q} \bigl( Y,Y^{q^D},\cdots,Y^{q^{D(s-1)}} \bigr) \]
  is a {\em nonzero} polynomial of degree at most $\Delta \cdot
  q^{D(s-1)}$.
\end{lemma}
\begin{proof}
If $\psi \in \L(\ell M')$ is nonzero, then $\psi(A') \neq
  0$. (Otherwise, the degree of zero divisor of $\psi$ will be at
  least $\deg(A') = b D > \ell d$, and thus exceed the degree of the
  pole divisor of $\psi$.) It follows that if $Q \neq 0$, then
  $\overline{Q}(Z_1,\dots,Z_s)$ obtained by reducing coefficients of
  $Q$ modulo $A'$ is also nonzero.\footnote{This is simplicity we gain
    by restricting the coefficients of $Q$ to also belong to $\L(\ell
    M')$.} Since the degree of $\overline{Q}$ in each $Z_i$ is at most
  $\Delta < q^D$, it is easy to see that $\Phi(Y) = \overline{Q}
  \bigl( Y,Y^{q^D},\cdots,Y^{q^{D(s-1)}} \bigr)$ is also nonzero. The
  degree of $\Phi$ is at $q^{D(s-1)}$ times the total degree of
  $\overline{Q}$, which is at most $\Delta$.
\end{proof}

By the above two lemmas, we see that one can compute the set of
residues $f(A')$ of all $f$ satisfying
$Q(f,\sigma_A(f),\dots,\sigma_{A^{s-1}}(f)) = 0$ by computing the
roots in $K_{A'}$ of $\Phi(Y)$. Since $\ev_{A'}$ is injective on $\L(\ell M')$ (Lemma~\ref{lem:ev-one-one}), this also lets us recover the message $f \in \L(\ell M')$.
\begin{lemma}
\label{lem:rf-algo}
  Given a nonzero polynomial $Q(Z_1,\dots,Z_s)$ with coefficients from
  $\L(\ell M')$ and degree $\Delta < q^D$, the set of functions
  \[ {\mathcal S} = \{ f \in \L(\ell M') \mid Q \bigl(
  f,\sigma_A(f),\dots,\sigma_{A^{s-1}}(f) \bigr) = 0\} \] has
  cardinality at most $q^{Ds}$.  

  Moreover, knowing the evaluations of a basis ${\mathcal B} = \{
  \phi_1,\phi_2,\dots,\phi_k\}$ of $\L(\ell M')$ at the place $A'$,
  one can compute the coefficients expressing each $f \in {\mathcal
    S}$ in the basis ${\mathcal B}$ in $q^{O(Ds)}$ time.
\end{lemma}
\begin{proof}
  As argued above, any desired $f \in \L(\ell M')$ has the property
  that $\Phi(f(A')) = 0$, so the evaluations of functions in
  ${\mathcal S}$ take at most ${\rm degree}(\Phi) \le \Delta
  q^{D(s-1)} \le q^{Ds}$ values. Since $\ev_{A'}$ is injective on
  ${\mathcal S}$, this implies $|{\mathcal S}| \le q^{Ds}$.
  The second part follows since we can compute the roots of $\Phi$ in
  $K_{A'}$ in time ${\rm poly}(q^{Ds},\log |K_{A'}|) \le
  q^{O(Ds)}$. Knowing $f(A')$, we can recover $f$ (in terms of the
  basis ${\mathcal B}$) by solving a linear system if we know the
  evaluations of the functions in the basis ${\mathcal B}$ at
  $A'$. The next section discusses a convenient representation for
  computations in $K_{A'}$.
\end{proof}

\subsubsection{Representation of the residue field $K_{A'}$}
The following gives a convenient representation for elements of
$K_{A'}$ which can be used in computations involving this field.
\begin{lemma}
\label{lem:residues-at-A'}
  The elements $\{1,\mu(A),\dots,\mu(A)^{b-1}\}$ form a basis for
  $K_{A'}$ over the field $R_T/(A) \simeq \F_{q^D}$. In other words, elements of
  $K_{A'}$ can be expressed in a unique way as 
\[ \sum_{i=0}^{b-1}
  b_i(T) \mu(A)^i \] where each $b_i \in R_T$ has degree less than $D$.
\end{lemma}
\begin{proof}
  Since $A$ is inert in $E/F$, the minimal polynomial $h(Z)$ of $\mu$
  over $F$ has the property that $\overline{h}(Z)$, obtained by
  reducing the coefficients of $h$ modulo $A$, is irreducible over the
  residue field $R_T/(A)$ . Thus $\mu(A)$ generates $K_{A'}$
  over $R_T/(A)$, and in fact minimal polynomial of $\mu(A)$ w.r.t to
  $K_A$ equals $\overline{h}(Z)$. Note that the coefficients of
  $\overline{h}$, which belong to $R_T/(A)$, have a natural
  representation as a polynomial in $R_T$ of degree $< \deg(A) = D$.
\end{proof}

We note that given the representation of the basis ${\mathcal B} = \{
\phi_1,\phi_2,\dots,\phi_k\}$ in the form guaranteed by
Theorem~\ref{thm:form-of-messages}, one can trivially compute the
evaluations of $\phi_i(A')$ in the above form. There is no need to
explicitly compute $\mu(A) \in \O_A/A$.  Therefore, the decoding
algorithm requires no additional pre-processed information beyond a
basis for the message space $\L(\ell M')$ --- the rest can all be
computed efficiently from the basis alone. 

\subsection{Wrap-up}
We are now ready to state our final decoding claim.
\begin{theorem}
\label{thm:main-cyclo}
For any $s$, $2 \le s \le m$, and $\zeta > 0$, the folded cyclotomic
code $\C \subseteq (\F_q^m)^{N}$ defined in (\ref{eq:code-def}) can be
list decoded in time $(Nm)^{O(1)}(s/\zeta)^{O(s)} + q^{O(Ds)}$ from a
fraction $\rho$ of errors
\begin{equation}
\label{eq:bd-on-errs-2}
\rho =  1 - (1+\zeta) \left( \frac{R_0  m}{m-s+1} \right)^{1-1/s} \left( 1 + \frac{d}{2R_0 r} \right) \ ,
\end{equation}
where $R_0 = k/n$ is the rate of the code. The size of the output list
is at most $q^{Ds}$. The decoding algorithm assumes polynomial amount
of pre-processed information consisting of basis functions
$\{\phi_1,\dots,\phi_k\}$ for the message space $\L(\ell M')$
represented in the form (\ref{eq:form-of-basis-fns}).  
(Note that this is the {\em same} representation used 
for encoding, and it is succinct by Lemma~\ref{lem:ub-on-degree}.)
\end{theorem}
\begin{proof}
  We first note that bound on fraction of errors follows from
  Theorem~\ref{thm:bd-on-errs}, and the fact that $k = R_0 n = R_0 N m
  = R_0 b r$. By Lemma~\ref{lem:interpolation} and its proof, in Step
  1 of the algorithm we can find a nonzero polynomial $Q$ (of degree
  $< q^D$) such that for any $f \in \L(\ell M')$ that needs to be
  output by the list decoder, we must have
  $Q(f,\sigma_A(f),\cdots,\sigma_{A^{s-1}}(f)) = 0$. We can evaluate
  the basis functions $\phi_i$ at $P_j^{(\beta)}$ in $(\ell
  q^d)^{O(1)}$ time by Lemma~\ref{lem:compute-mu-vals}, and with this
  information, the running time of this interpolation step can be
  bounded by $(Nm)^{O(1)} (w+s)^{O(s)} = (Nm)^{O(1)}
  (s/\zeta)^{O(s)}$ (since $w = O(s/\zeta)$). We can also efficiently compute the evaluations
  of $\phi_i$ at $A'$ in the representation suggested by
  Lemma~\ref{lem:residues-at-A'}.  Therefore, by
  Lemma~\ref{lem:rf-algo}, we can then find a list of the at most
  $q^{Ds}$ functions $f$ satisfying
  $Q(f,\sigma_A(f),\cdots,\sigma_{A^{s-1}}(f)) = 0$ in $q^{O(Ds)}$
  time.
\end{proof}

\begin{remark}[List Recovery]
\label{rem:lr}
A similar claim holds for the more general {\em list recovery}
problem, where for each position we are given as input a set of up to
$l$ elements of $\F_q^m$, and the goal is to find all codewords which
agree with some element of the input sets for at least a fraction
$(1-\rho)$ of positions. In this case, $1-\rho$ only needs to be only
a factor $l^{1/s}$ larger than the bound (\ref{eq:bd-on-errs-2}). By
picking $s \gg l$, the effect of $l$ can be made negligible. This
feature is very useful in concatenation schemes; see
Section~\ref{sec:bin-zyablov} and \cite{GR-capacity} for further
details.
\end{remark}

\section{Long codes achieving list decoding capacity}
\label{sec:params}
We now describe the parameter choices which leads to
capacity-achieving list-decodable codes, i.e., codes of rate $R_0$
that can correct a fraction $1-R_0-\eps$ of errors (for any desired $0
< R_0 < 1$), and whose alphabet size is polylogarithmic in the block
length; the formal statement appears in Theorem~\ref{thm:main-final}
below. (Recall that for folded RS codes, the alphabet size is a large
polynomial in the block length.)  Using concatenation and
expander-based ideas, Guruswami and Rudra~\cite{GR-capacity} also
present capacity-achieving codes over a fixed alphabet size (that
depends on the distance $\eps$ to capacity alone). The advantage of
our codes is that they inherit strong list recovery properties similar
to the folded RS codes (Remark~\ref{rem:lr}). This is very useful in
concatenation schemes, and indeed our codes can be used as outer codes
for an explicit family of binary concatenated codes list-decodable up
to the Zyablov radius, {\em with no brute-force search} for the inner
code (see Section~\ref{sec:bin-zyablov} below).

We now describe our main result on how to obtain the desired codes
from the construction $\C$ and Theorem~\ref{thm:main-cyclo}. The
underlying parameter choices to achieve this require a fair bit of
care.

\iffalse
%% appendix pointer
Due to space restrictions, the proof of the following theorem appears
in Appendix~\ref{app:pfs-omit-6}.  
\fi

\begin{theorem}[Main]
\label{thm:main-final}
For every $R_0$, $0 < R_0 < 1$, and every constant $\eps > 0$, the
following holds for infinitely many integers ${\mathbf q}$ which are powers of two. There
is a code of rate at least $R_0$ over an alphabet of size ${\mathbf
  q}$ with block length 
$N \ge 2^{{\mathbf q}^{\Omega(\eps^2/\log(1/R_0))}}$ 
%$N \ge {\mathbf q}^{\Omega(\eps^3 {\mathbf q}^{\eps^2/2})}$ 
that can be list decoded up to a fraction $1-R_0 -\eps$ of errors in
time bounded by $(N \log(1/R_0)/\eps^2)^{O(1/(R_0 \eps)^2)}$.
\end{theorem}
\begin{proof}
  Suppose $R_0$, $0 < R_0 < 1$, and $\eps > 0$ are given.  Let $c = 2
  \lfloor \frac{10}{R_0 \eps} \rfloor + 1$, and $\phi(c)$ denote the
  Euler's totient function of $c$.

Let $u \ge 1$ be an arbitrary integer; we will get a family of codes
by varying $u$. The code we construct will be a folded cyclotomic code $\C$
defined in Eq. (\ref{eq:code-def}). Let $x = \phi(c) u$. Note that
$2^x \equiv 1 \pmod c$. We first pick $q,r,d$ as follows: $r = 2^x$,
$q=r^2$, and $d = (2^x-1)/c$. For this choice, $d | r-1$ and
$(q-1)/(r-1) = r+1$ is coprime to $d$, as required in
Lemma~\ref{lem:binomial-irred}. So we can take $M(T) = T^d - \gamma
\in \F_r[T]$ for $\gamma$ primitive in $\F_r$ as the irreducible
polynomial over $\F_q$.

For the above choice $d/r < 1/c \le \eps R_0/20$, so that
$\frac{d}{2R_0r} < \frac{\eps}{10}$. By picking 
\[ s = \Theta(\eps^{-1}
\log(1/R_0)), \quad  m = \Theta(s/\eps) \ , \]
and $\zeta = \eps/20$, we can
ensure that the decoding radius $\rho$ guaranteed in Eq.
(\ref{eq:bd-on-errs-2}) by Theorem~\ref{thm:main-cyclo} is at least $1- (1+\eps) R_0$.

The degree $b$ of the extension $E/F$ (Eq. (\ref{eq:def-of-b})) is given by
$b = \frac{r^d+1}{r+1}$. The length of the unfolded cyclotomic
code $\C^0$ (defined in (\ref{eq:basic-cycl})) equals $n = rb > r^d/2$.
We need to ensure that the rate of $\C^0$, which is equal to the rate of 
the folded
cyclotomic code $\C$, is at least $R_0$. To this end, we will pick 
\begin{equation}
  \ell = \left\lceil \frac{b}{2} + \frac{R_0 r b}{d} \right\rceil \ . 
\end{equation}
It is easily checked that for our choice of parameters $\ell \ge
b$. By Lemma~\ref{lem:params-of-unfolded-code}, the rate of $\C^0$
equals $\frac{d(\ell - (b-1)/2)}{rb}$, which is at least $R_0$ for the
above choice of $\ell$.

We next pick the value of $D$, the degree of the irreducible $A$,
which is the key quantity governing the list size and decoding
complexity. We need $D > \ell d/ b$. For the $\ell$ chosen above, this
condition is surely met if $D > 2r$. But there must also be an
irreducible $A$ of degree $D$ that is a primitive root modulo
$M$. Since we know the Riemann hypothesis for function fields, there
is an effective Dirichlet theorem on the density of irreducibles in
arithmetic progressions (see \cite[Thm 4.8]{rosen}). This implies that
when $D \gg 2d$, such a polynomial $A$ must exist (in fact about a
$\frac{\phi(q^d-1)}{D(q^d-1)}$ fraction of degree $D$ polynomials
satisfy the needed property). We can thus pick 
\[ D = \Theta(r) =
\Theta(dc) = \Theta(d/(R_0 \eps)) \ . \]

The running time of the list decoding algorithm is dominated by the
$q^{O(Ds)}$ term, and for the above choice of parameters can be
bounded by $q^{O(d/(R_0 \eps)^2)}$. The block length of the code $N$ satisfies 
\[ N = \frac{n}{m} > \frac{r^d}{2m} = 
\frac{q^{d/2}}{2m} = \Omega\left( \frac{\eps^2 q^{d/2}}{ \log(1/R_0)}\right) \ . \]
As a function of $N$, the decoding complexity is therefore bounded by
$(N \log(1/R_0)/\eps^2)^{O(1/(R_0 \eps)^2)}$.
% (suppressing the dependence on $R_0$ which we treat as a fixed
% constant).
The alphabet size of the folded
cyclotomic code is ${\mathbf q} = q^m$, and we can bound the block length $N$ from below as a function of ${\mathbf q}$ as:
\begin{eqnarray*}
  N & \ge &  \frac{q^{d/2}}{2m} \ge \frac{q^{\Omega(r/c)}}{2m} \ge \frac{q^{\Omega(\eps R_0 \sqrt{q})}}{2m} \\  
    & \ge & 2^{\sqrt{q}} \qquad \mbox{(for large enough $q$ compared to $1/R_0$, $1/\eps$)} \\
    & = & 2^{{\mathbf q}^{1/(2m)}} \ge 2^{{\mathbf q}^{\Omega(\eps^2/\log(1/R_0)))}} \ .
\end{eqnarray*}
This establishes the claimed lower bound on block length, and completes the proof of the theorem.
% The latter quantity is at least $2^{ {\mathbf q}^{\Omega(d/m)} \ge
%   {\mathbf q}^{\Omega(\eps \sqrt{q}/m)}$.  Since $m =
%   \Theta(\log(1/R_0)/\eps^2)$, it follows that we have established
%   the claim of the theorem.
\end{proof}

\subsection{Concatenated codes list-decodable up to Zyablov radius}
\label{sec:bin-zyablov}
Using the strong list recovery property of folded RS codes, a
polynomial time construction of binary codes list-decodable up to the
Zyablov radius was given in \cite[Thm 5.3]{GR-capacity}. The
construction used folded RS codes as outer codes in a concatenation
scheme, and involved an undesirable brute-force search to find a
binary inner code that achieves list decoding capacity. The time to
construct the code grew faster than $N^{\Omega(1/\eps)}$ where $\eps$
is the distance of the decoding radius to the Zyablov
radius. This result as well as our result below hold not only
  for binary codes but also codes over any fixed alphabet; for sake of
  clarity, we state results only for binary codes.

Since the folded cyclotomic codes from Theorem~\ref{thm:main-final}
are much longer than the alphabet size, by using them as outer codes,
it is possible to achieve a similar result without having to search
for an inner code, by using as inner codes {\em all possible binary
  linear codes} of a certain rate!

\begin{theorem}
\label{thm:bin-zyablov}
Let $0 < R_0,r < 1$ and $\eps > 0$. Let $\C$ be a folded cyclotomic
code guaranteed by Theorem~\ref{thm:main-final} with rate at least
$R_0$ and a large enough block length $N$. Let $\C^*$ be a binary code
obtained by concatenating $\C$ with all possible binary linear maps of
rate $r$ (each one used a roughly equal number of times). Then $\C^*$
is binary linear code of rate at least $R_0\cdot r$ that can be list
decoded from a fraction $(1-R_0) H^{-1}(1-r) -\eps$ of errors in
$N^{(1/\eps)^{O(1)}}$ time.
\end{theorem}

We briefly discuss the idea behind proving the above claim.  As the alphabet
size of folded cyclotomic codes is polylogarithmic in $N$, each outer
codeword symbol can be expressed using $O_{\eps}(\log \log N)$
bits. Hence the total number of such inner codes $S$ will be at most
$2^{O_{\eps}((\log \log N)^2)}\ll N$ for large enough $N$. The $N$
outer codeword positions will be partitioned into $S$ (roughly) equal
parts in an arbitrary way, and each inner code used to encode all the
outer codeword symbols in one of the parts. Most of the inner codes
achieve list decoding capacity --- if their rate is $r$, they can list
decode $H^{-1}(1-r)-\eps$ fraction of errors with constant sized lists
(of size $2^{O(1/\eps)}$). This suffices for analyzing the standard
algorithm for decoding concatenated codes (namely, list decode the
inner codes to produce a small set of candidate symbols for each
position, and then list recover the outer code based on these
sets). Arguing as in \cite[Thm 5.3]{GR-capacity}, we can thus prove
Theorem~\ref{thm:bin-zyablov}.

\section*{Acknowledgments}
\noindent Many thanks to Dinesh Thakur for several illuminating
discussions about Carlitz-Hayes theory and cyclotomic function
fields. I thank Dinesh Thakur and Greg Anderson for helping me with
the proof of Lemma~\ref{lem:ub-on-degree}. Thanks to Andrew Granville
for pointing me to Dirichlet's theorem for polynomials.

%\vspace{0.1in}
%\bibliographystyle{abbrv} \bibliography{cyclo}

%\newpage
\appendix

\section{Table of parameters used}
\label{app:params}
\noindent
Since the construction of the cyclotomic function field and the associated error-correcting code used a large number of parameters, we summarize them below for easy reference.

\smallskip \noindent We begin by recalling the parameters concerning the function field construction: \\

\begin{tabular}{ll}
$q$ & size of the ground finite field \\
$r$ & size of the subfield $\F_r \subset \F_q$ \\
$F$ & the field $\F_q(T)$ of rational functions \\
$R_T$ & the ring of polynomials $\F_q[T]$ \\
$P_\infty$ & the place of $F$ that is the unique pole of $T$ \\
$M$ & polynomial $T^d - \gamma \in \F_r[T]$, irreducible over $\F_q$ \\
$d$ & degree of the irreducible polynomial $M$ \\
$C_M$ & the Carlitz action corresponding to $M$ \\
$\Lambda_M$ & the $M$-torsion points in $F^{\ac}$ under the action $C_M$ \\
$K$ & the cyclotomic function field $F(\Lambda_M)$ \\
$\lambda$ & nonzero element of $\Lambda_M$ that generates $K$ over $F$; $K = F(\lambda)$ \\
$G$ & the Galois group of $K/F$, naturally isomorphic to $(R_T/(M))^*$ \\
$H$ & the subgroup $\F_q^* \cdot \F_r[T]$ of $G$ \\
$E$ & the fixed field $K^H$ of $H$ \\
$\mu$ & primitive element for $E/F$; $E = F(\mu)$ \\
$b$ & the degree $[E:F]$ of the extension $E/F$ \\
$g$ & the genus of $E/F$, equals $d(b-1)/2+1$  \\
\end{tabular}

\medskip \noindent
The construction of the code $\C^0$ (Eqn. (\ref{eq:basic-cycl})) and its folded version $\C$ (Eqn. (\ref{eq:code-def})) used further parameters, listed below: \\

\begin{tabular}{ll}
  $M'$ & the unique place of $E$ lying above $M$ \\
  $\ell$ & maximum pole order at $M'$ of message functions; $\ell \ge b$ \\
  $\L(\ell M')$ & $\F_q$-linear space of messages of the codes \\
  $n$ & block length of $\C^0$, $n = br$ \\
  $k$ & dimension of the $\F_q$-linear code $\C$, $k = \ell d  - g +1$ \\
  $m$ & folding parameter \\
  $N$ & block length of folded code $\C$, $N = n/m$ \\
  $P^{(\beta)}_j$ & for $\b \in \F_r$ and $0 \le j < b$, these are the rational places lying above $T-\beta$ in $E$ \\
$A$ & an irreducible polynomial (place of $F$) that remains inert in $E/F$ \\
$D$ & the degree of the polynomial $A$; satisfies $Db >\ell d$ \\
$\sigma_A$ & the Artin automorphism of the extension $E/F$ at $A$ \\
$A'$ & the unique place of $E$ lying above $A$ \\

\end{tabular}

\section{Algebraic preliminaries}
\label{app:alg-prelims}

We review some basic background material concerning global fields and
their extensions. The term global field refers to either a number
field, i.e., a finite extension of $\Q$, or the function field $L$ of
an algebraic curve over a finite field, i.e., a finite extension of
$F =\F_q(T)$. While we are only interested in the latter, much of the
theory applies in a unified way to both settings. Good references for
this material are the texts by Marcus~\cite{marcus} and
Stichtenoth~\cite{stich-book}.

\subsection{Valuations and Places}
A subring $X$ of $L$ is said to be a {\em valuation ring} if for every
$z \in L$, either $z \in X$ or $z^{-1} \in X$. Each valuation ring is
a {\em local ring}, i.e., it has a unique maximal ideal.  The set of
{\em places} of $L$, denoted $\P_L$, is the set of maximal ideals of
all the valuation rings of $L$. Geometrically, this corresponds to the
set of all (non-singular) points on the algebraic curve corresponding
to $L$.  The valuation ring corresponding to a place $P$ is called the
ring of {\em regular functions} at $P$ and is denoted $\O_P$.

Associated with a place $P$ is a {\em valuation} $v_P : L \rightarrow
{\mathbb Z} \cup \{\infty\}$, that measures the order of zeroes or
poles of a function at $P$, a negative valuation implies the function
has a pole at $P$ (by convention we set $v_P(0) =\infty$). In terms of
$v_P$, we have $\O_P = \{ x \in L \mid v_P(x) \ge 0\}$ and $P = \{x
\in L \mid v_P(x) > 0\}$. The valuation $v_P$ satisfies $v_P(xy) =
v_P(x) + v_P(y)$ and the triangle inequality $v_P(x+y) \ge \min\{
v_P(x), v_P(y)\}$ (and equality holds if $v_P(x) \neq v_P(y)$). 

The quotient $\O_P/P$ is a field since $P$ is a maximal ideal and it is
called the {\em residue field} at $P$.  The residue field $\O_P/P$ is a
finite extension field of $\F_q$; the degree of this extension is
called the {\em degree} of $P$. We will also sometimes use the
terminology {\em primes} to refer to places --- the terms primes and
places will be used interchangeably.

\subsection{Decomposition of primes in Galois extensions}
We now discuss how primes decompose in field extensions. Let $K/L$ be
a finite, separable extension of global fields of degree $[K:L] =
  n$. We will restrict our attention of Galois extensions. Let $P$ be
  a place of $L$. Let $\O'_{P}$ be the integral closure of $\O_P$ in
  $K$, i.e., the set of all $z \in K$ which satisfy a monic polynomial
  equation with coefficients in $\O_P$. The ideal $P \O'_{P}$ can be
  written as the product of prime ideals of $\O'_{P}$ as $P \O'_{P} =
  (P_1 P_2 \dots P_r)^e$. Here $P_1,P_2,\dots, P_r$ are said to be the
  places of $K$ lying above $P$ (and $P$ is said to be lie below each
  $P_i$). One has the equality $P_i \cap L = P$ for every $i$. The ring
  $\O'_{P}$ is the fact the intersection of $\O_{P_i}$ for
  $i=1,2,\dots,r$. The quantity $e$ is called the {\em ramification
    index}, and when $e = 1$, $P$ (as well as the $P_i$) are said to
  be {\em unramified}. For $x \in L$, one has $v_{P_i}(x) = e \cdot
  v_P(x)$.  The residue field $\O_{P_i}/P_i$ is a finite extension of
  $\O_P/P$; the degree $f$ of this extension is called the inertia
  degree of $P$. The ramification index $e$, inertia degree $f$, and
  number $r$ of primes above $P$ satisfy $efr = n = [K:L]$.

  If $e=n$ and $f=r=1$, the prime $P$ is said to be {\em totally
    ramified}. If $r=n$ and $e=f=1$, the prime $P$ is said to {\em
    split completely}. If $f=n$ and $e=r=1$, the prime $P$ is said to
  be {\em inert}.

\subsection{Galois action on primes and the Artin automorphism}

The Galois group $G=\Gal(K/L)$ acts transitively on the primes
$P_1,P_2,\dots,P_r$ of $K$ lying above $P \in \P_L$. For each $P_i$,
there is a subgroup $D(P_i|P) \subseteq G$ that fixes $P_i$; this is
called the {\em decomposition group} of $P_i$. It is known that the
decomposition is isomorphic to the Galois group of the finite field
extension $(\O_{P_i}/P_i)/(\O_P/P)$ of the residue fields. Note that
the latter group is cyclic and generated by the Frobenius automorphism
${\rm Frob}$ mapping $x \mapsto x^q$. The element of $D(P_i|P)$
corresponding to ${\rm Frob}$ is called the Artin automorphism
${\mathcal A}(P_i|P)$ of $P_i$ over $P$.

When $G$ is abelian (which covers the cases we are interested in), the
decomposition group $D(P_i|P)$ and the Artin automorphism ${\mathcal
  A}(P_i|P)$ are the same for every $P_i$, and they depend only on the
prime $P$ below. Denote the Artin automorphism at $P$ by ${\mathcal A}_P$. This has the following important property:
\[ {\mathcal A}_P(x) \equiv x^{\|P\|} \pmod {P_i} \] for every $x \in
\O'_{P}$ and every prime $P_i$ lying above $P$. If $P$ is unramified, then
${\mathcal A}_P$ is the only element of $G$ with this property. In the unramified case, by Chinese Remaindering the above also implies 
\[ {\mathcal A}_P(x) \equiv x^{\|P\|} \pmod {P \O'_{P}} \]
for every $x \in \O'_{P}$.

Note that if $P$ is inert with a unique prime $P'$ lying above it,
then $D(P'|P) = G$, and thus $G$ must be cyclic. Thus, only cyclic
extensions can have an inert prime.

\medskip

\end{document}